\documentclass[10pt]{amsart}

\usepackage{mathtools}
\usepackage[round]{natbib}
\setcitestyle{numbers,square}
\usepackage{hyperref}
\usepackage{xcolor}
\usepackage[normalem]{ulem}

\usepackage{accents}

\newcommand\redsout{\bgroup\markoverwith{\textcolor{red}{\rule[0.5ex]{2pt}{0.4pt}}}\ULon}

\newcommand{\E}{\mathbb{E}}
\newcommand{\N}{\mathbb{N}}
\newcommand{\Z}{\mathbb{Z}}
\newcommand{\Q}{\mathbb{Q}}
\newcommand{\R}{\mathbb{R}}

\newcommand{\Pb}{\mathbb{P}}
\newcommand{\tp}{t^{\prime}}
\newcommand{\pa}{a^{\prime}}

\newcommand{\ps}{s^{\prime}}

\newcommand{\xp}{x^{\prime}}

\newcommand{\ve}{\varepsilon}

\newcommand{\hno}{\frac{\partial^{}h_n}{\partial x}}
\newcommand{\hnt}{\frac{\partial^{2}h_n}{\partial x^2}}
\newcommand{\htt}{\frac{\partial^{3}h_n}{\partial x^3}}

\newcommand{\hsn}{\frac{\partial^{2}h_n}{\partial s\partial x}}
\newcommand{\hnn}{\frac{\partial^{3}h_n}{\partial s\partial x^2}}
\newcommand{\hpo}{\dfrac{\partial^{}h}{\partial x}}
\newcommand{\hpp}{\dfrac{\partial^{2}h}{\partial x^2}}
\newcommand{\htp}{\dfrac{\partial^{3}h}{\partial x^3}}

\newcommand{\hps}{\dfrac{\partial^{}h}{\partial s}}
\newcommand{\hsp}{\dfrac{\partial^{2}h}{\partial s\partial x}}

\def\={{\;\mathop{=}\limits^{\text{(law)}}\;}}

\newtheorem{theorem}{Theorem}[section]
\newtheorem{prop}[theorem]{Proposition}
\newtheorem{lemma}[theorem]{Lemma}
\newtheorem{defi}[theorem]{Definition}
\newtheorem{corollary}[theorem]{Corollary}

\newtheorem{remark}[theorem]{Remark}

\numberwithin{equation}{section}

\usepackage[english]{babel}

\title[Regularising properties of Brownian sheet paths]
{Stochastic integration with respect to local time of the Brownian sheet and regularising properties of Brownian sheet paths}

\author[A.-M. Bogso]{Antoine-Marie Bogso}
\address{University of Yaounde I\\
	Faculty of Sciences, Department of Mathematics\\
	P.O. Box 812, Yaounde, Cameroon \\
	and African institute for Mathematical Sciences Ghana, P.O. Box LGDTD 20046, Summerhill Estates, East Legon Hills, Santoe, Acrra}
\email{antoine.bogso@facsciences-uy1.cm, antoine@aims.edu.gh}           

\author[M. Dieye]{Moustapha Dieye}

\address{African institute for Mathematical Sciences Ghana, P.O. Box LGDTD 20046, Summerhill Estates, Eat Legon Hills, Santoe, Acrra}

\email{moustapha@aims.edu.gh}

\author[O. Menoukeu Pamen]{Olivier Menoukeu Pamen}
\address{Institute for Financial and Actuarial Mathematics (IFAM) \\
	Department of Mathematical Sciences, University of Liverpool \\
	Liverpool L69 7ZL, UK \\
	and AIMS Ghana}
\email{menoukeu@liverpool.ac.uk}
\thanks{The project on which this publication is based has been carried out with funding provided by the Alexander von Humboldt Foundation, under the programme financed by the German Federal Ministry of Education and Research entitled German Research Chair No 01DG15010.}
\subjclass{Primary 60G17}

\keywords{Brownian sheet, SDEs on the plane, path by path uniqueness}

\date{\today}

\begin{document}

	\begin{abstract}
		In this work, we generalise the stochastic local time space integration introduced in \cite{Ei00} to the case of Brownian sheet. 
			This allows us to prove a generalised two-parameter It\^o formula and derive Davie type inequalities for the Brownian sheet. Such estimates are useful to obtain regularity bounds for some averaging type operators along Brownian sheet curves.    
	\end{abstract}
	
	\maketitle
	 
		\section{Introduction}
	The goal of this paper is three fold: we first extend the stochastic local time-space calculus introduced by Eisenbaum \cite{Ei00} to the multidimensional standard Brownian sheet. We then obtain the It\^o formula under weaker condition on the function. Finally, we derive several estimate of the ``averaging type operator" introduced in \cite{TW03} and further studied in \cite{CG16}. 
	
	The notion of integration with respect to space-time local time to our knowledge was introduced in \cite{Ei00}. In that work, the author defined a stochastic integral of Borel measurable functions on $\R_+\times\R$ with respect to the local time process of a linear Brownian motion. This allowed the author to extend the It\^o formula to a large class of differentiable random functions with locally bounded derivatives and to define the local time of Brownian motion on any Borelian curve. The previous result were extended in \cite{Ei06} to the case of L\'evy process and reversible semimartingales in \cite{Ei07}. This powerful tool led to many interesting generalisations of the It\^o formula (see for example \cite{ElTrZh07, FeZh07, peskir05} and references therein). 
	
	In order to write a local time-space integral in the two-parameter setting, we consider the local time of the Brownian sheet with respect to the Lebesgue measure on $\R_+^2$ defined by Walsh \cite{Wa78}. Observe that the Lebesgue measure on $\R_+^2$ is the measure induced by the quadratic variation of the Brownian sheet. As in \cite{Ei00}, this local time can be expressed as the sum of a forward and a backward It\^o integral. Using a representation of the backward It\^o integral, we define the stochastic integral with respect to the local time for elements of a Banach space. This enables us to write a counterpart of Eisenbaum's local time-space integration formula and a generalised It\^o formula for the Brownian sheet. A key step in proving this result is the representation formula of the reversal process of the Brownian sheet in one parameter at a fixed time obtained by Dalang and Walsh \cite[Theorem 6.1]{DaW02}. Let us mention that Sanz \cite{Sa88} took advantage of the ideas developed in \cite{Wa78} to define a notion of local time for a class of continuous two-parameter martingales with respect to the quadratic variation.

	It is worth mentioning that there exists another notion of local time of the Brownian sheet $(W_{s,t},(s,t)\in \R_+^2)$ with respect to the measure on $\R_+^2$ induced by quadratic variation $\langle J\rangle$ of the martingale $J=(J_{s,t},(s,t)\in \R_+^2)$ given by
	\begin{align*}
		J_{s,t}=\int_0^t\int_0^s\int_0^t\int_0^s\mathbf{1}_{\{u<v,\xi>\zeta\}}\mathrm{d}W_{u,\xi}\mathrm{d}W_{v,\zeta}.
	\end{align*}
	This concept was introduced in \cite{CW75} and led to the first results on local time for two-parameter processes. Walsh \cite{Wa78} points the difference between the two definitions of local time. A complete study of local time with respect to multi-parameter analogue of $\langle J\rangle$ for the multi-parameter $\R^d$-valued Brownian sheet was given by Imkeller \cite{Im84}. The author also established a multi-parameter stochastic calculus and a multiparameter It\^o-Tanaka formulas (see \cite{Im84b,Im84c}). Nualart \cite{Nu84b} proved existence of a local time with respect to the measure induced by $\langle J\rangle$ for continuous fourth-power integrable (two-parameter) martingales that vanish on the boundary of $\R_+^2$.

	Finally, we consider the following average type transfrom 
	averaging transforms of type  
	\begin{align}\label{eqav1}
		T_{I}^{W}[b](s,x)=\int_{I}b(t,x+W_{s,t})dt,\,\forall\,(s,x)\in\R_+\times\R^d,
	\end{align}
	where $I$ is a finite sub-interval of $\R_+$,
	$b:\,\R_+\times\R^d\to\R^d$ is a 
	bounded Borel measurable function  
	and $(W_{s,t},(s,t)\in \R^2_+)$ is a $\R^d$-valued Brownian sheet given on some filtered probability space $(\Omega,\mathcal{F},(\mathcal{F}_{s,t},(s,t)\in \R_+^2),\Pb)$.	The terminology {\it averaging operator} is borrowed from Catellier and Gubinelli \cite{CG16} (see also Galeati and Gubinelli \cite{GG21}) where the authors obtain regularising estimates of 
	the averaging transform along the paths of the $d$-dimensional fractional Brownian motion $(B^H_t,t\geq0)$ with Hurst parameter $H$ defined by
	\begin{align*}
		T^{B^{H}}_t[b](x)=\int_0^tb(x+B^{H}_s)ds,\,\forall\,(t,x)\in\R_+\times\R^d
	\end{align*}
	and take advantage of these estimates to establish existence and uniqueness of solution to the ordinary differential equation (ODE) in $\R^d$, $d\in\N$,
	\begin{align}\label{eqODE}
		\left\{
		\begin{array}{ll}
			\dot{x}(t)=b(t,x(t))+\dot{w}(t),&t\in\R_+,\\
			x(0)=x_0,
		\end{array}
		\right.
	\end{align}	
	where $x,w\in\mathcal{C}([0,1],\R^d)$, $b$ is a time-dependent vector field which may be only a distribution in the space variable and the dot denotes differentiation with respect to time. These results are counterparts of those obtained by Davie \cite{Da07} for the ODE \eqref{eqODE}, where $w$ is a $\R^d$-valued Brownian path and $b$ is a bounded Borel measurable function. One key result in \cite{Da07} is that if $b:\,[0,1]\times\R^d\to\R^d$ is a bounded Borel measurable function, then, for almost all $\R^d$-valued Brownian paths $w$, the function $T^B_t[b]$ is almost Lipschitz continuous with a modulus of continuity of the type $|x|\log^{1/2}(1/|x|)$.	To be more precise, consider for a moment the integral form of ODE \eqref{eqODE} given by
	\begin{align}\label{eqODE2}
		x_t=x_0+\int_0^tb(s,x_s)ds+w_t,\quad t\in\R_+.
	\end{align}
	Setting $w^{x_0}=x_0+w$ and $y_t=x_t-w_t-x_0$, Equation \eqref{eqODE2} can be rewritten as
	\begin{align}\label{eqODE3}
		y_t=\int_0^tb(s,y_s+x_0+w_s)ds=T_t^{w^{x_0}}[b](y),\quad t\in\R_+, 
	\end{align}
	where $w^{x_0}=x_0+w$ and $y_t=x_t-w_t-x_0$. Any solution $y$ of \eqref{eqODE3} is a fixed point of the map $z\longmapsto T_{\cdot}^{w^{x_0}}[b](z)$ from $\mathcal{C}(\R_+,\R^d)$ to itself. The existence and uniqueness of such a fixed point rely on specific regularity properties of the operator $T_{\cdot}^{w^{x_0}}[b]$. Davie \cite{Da07} exploited the almost Lipschitz regularity of the averaging operator to prove uniqueness of fixed point for any $\R^d$-valued Brownian path $w$ in a set of full mass. Catellier and Gubinelli \cite{CG16} took advantage of almost sure continuity of $T_{\cdot}^{w^{x_0}}[b]$ for a large class of H\"older-Besov distributions $b$ and a set of fractional Brownian perturbations of full mass to prove existence and uniqueness of fixed point. Galeati and Gubinelli \cite{GG21} provided smoothness conditions on $T_{\cdot}^{w^{x_0}}[b]$, under which the ODE admits a flow with prescribed regularity. They also establish well-posedness for certain perturbed transport type partial differential equations (PDEs) under suitable smoothness properties of the averaging map. Chouk and Gubinelli \cite{ChG15} analysed the regularising properties of fractional Brownian paths $w$ in terms of the averaging operator $T^w_t$ in the context of non-linear dispersive PDEs modulated by an irregular signal. In particular, they obtained global well-posedness for the modulated Non-linear Shr\"ondinger equation with generic power nonlinearity.

	The averaging transform $T^{W}$ is the convolution with the curve of a Brownian sheet when one time parameter is fixed, and can also be seen as the convolution against the measure induced by the local time process $(L_x^1(s,t);x\in\R,t\in\R_+)$ of the Brownian motion $(W_{s,t},t\in\R_+)$ (confer \cite{CG16}). More precisely, by an occupation time formula given in \cite[Eq. (2.6)]{Wa78},
	\begin{align*}
		T^W_I[b](s,x)=\int_Ib(t,x+W_{s,t})dt=\int_I\int_{\R}b(t,x+y)d_tL^y_1(s,t)dy,
	\end{align*}
	where $I$ denotes a finite interval of $\R$.
	
	We wish to derive some regularity estimates of the averaging operator given by \eqref{eqav1} (see Theorem \ref{theo:Pseudometric} and Corollary \ref{corol:Pseudometric}). These estimates play a key role in the study of  the path-by-path uniqueness of solutions to the following hyperbolic differential equation 
	\begin{align}\label{hspde}
		\left\{	\begin{array}{ll}
			\dfrac{\partial^2x(s,t)}{\partial s\partial t}=b(t,x(s,t))+\dfrac{\partial^2W_{s,t}}{\partial s\partial t}\\ \\
			x(0,t)=x_0=x(s,0),
		\end{array}
		\right.
	\end{align}
	It can easily be shown that the path-by-path uniqueness of the above equation \eqref{hspde} is equivalent to that of the following integral equation:
	\begin{align*}
		y(s,t)=\int_0^s\int_0^tb(t_1,y(s_1,t_1)+x_0+W_{s_1,t_1})\,\mathrm{d}t_1\mathrm{d}s_1=\int_0^sT^{W^{x_0}}_{[0,t]}[b](y)\,\mathrm{d}s_1.
	\end{align*}
	Such equation was studied in \cite{BDM22} by the same authors. More precisely, they show that the path-by-path uniqueness of \eqref{hspde} is valid when the drift is componentwise nondecreasing and satisfies spatial linear growth condition (see \cite[Theorem 3.2]{BDM22}). In addition, since path-by-path uniqueness implies pathwise uniqueness (confere \cite[ Section 1.8.5]{BFGM19}), it follows from a Yamada-Watanabe type result for Brownian-sheet (see for example \cite{NuYe89}) that equation \eqref{hspde} has a unique strong solution (see \cite[Corollary 3.3]{BDM22}). The results obtained in \cite{BDM22} were generalised in \cite{BMP22a} when the drift is the difference of two componentwise monotone functions and staisfying the linear growth condition (see \cite[Theorems 2.7 and 2.8]{BMP22a}). In addition, it was proved that the obtained solution is Malliavin differentiable; see \cite[Theorem 3.4.]{BMP22a}for bounded drift and \cite[Theorem 3.13]{BMP22a} for drift satisfying spatial linear growth condition). These results constitute a big improvement in this direction since to the best of our knowledge there no results in this direction under such conditions. The case of path-by-path uniquness of solution to \eqref{hspde} when the drift $b$ is merely measurable and bounded is at the moment still open. One difficulty relies on the choice of a convenient Euler-Maruyama type scheme that would leads to an additional useful regularising estimates which is key to obtaining a Gronwal type lemma (compare with \cite[Lemma 3.9]{BDM22}). When the noise is replaced by a fractional Brownian sheet, one of the main difficulty would be the definition of a two-parameter Young integral, which is central in proving the regularity of the averaging operator.

	There are many results in the analysis on weak and strong-type estimates for convolutions with deterministic curves. These operators are investigated for instance in \cite{TW03}, where the authors establish strong-type $(L^p,L^q)$ in the interior of a trapezoid and failure of a restricted type $(L^p,L^q)$ outside the same trapezoid, when $1\leq p<q\leq\infty$. Futher results on estimates for convolutions with deterministic curves may be found in \cite{Cr98,GS98,GSW99,Li73,Ob87,Ob99}.

	We now give rigourous definitions of filtered probability space and $\R^d$-valued Brownian sheet indexed by $\R_+^2$ with respect to a system of $\sigma$-algebras. We endow $\R_+^2$ with the following partial order 
	$$
	(s,t)\preceq(\ps,\tp)\text{ when }s\leq \ps\text{ and }t\leq\tp.
	$$
	We also write
	$$
	(s,t)\prec(\ps,\tp)\text{ when }s<\ps\text{ and }t<\tp.
	$$
	The next definitions may be found in \cite[Section 1]{NuYe89}.
	\begin{defi}
		Let $(\Omega,\mathcal{F},\Pb)$ be a probability space and 
		$(\mathcal{F}_{s,t})_{(s,t)\in\R^2_+}$ be a system of sub-$\sigma$-algebras of $\mathcal{F}$. We say that $(\Omega,\mathcal{F},(\mathcal{F}_{s,t},(s,t)\in \R_+^2),\Pb)$ is a filtered probability space if
		\begin{enumerate}
			\item $(\Omega,\mathcal{F},\Pb)$ is a complete probability space;
			\item $(\mathcal{F}_{s,t},(s,t)\in \R_+^2)$ is a non-decreasing system in the sense that $\mathcal{F}_{s,t}\subset\mathcal{F}_{(\ps,\tp)}$ when $(s,t)\preceq(\ps,\tp)$;
			\item $\mathcal{F}_{(0,0)}$ contains all the null sets in $(\Omega,\mathcal{F},\Pb)$,
			\item $(\mathcal{F}_{s,t},(s,t)\in \R_+^2)$ is a right-continuous system in the sense that
			$$
			\mathcal{F}_{s,t}=\bigcap\limits_{(s,t) \prec (\ps,\tp)}\mathcal{F}_{\ps,\tp}.
			$$
		\end{enumerate}
		We call filtration any non-decreasing system of sub-$\sigma$-algebras of $\mathcal{F}$. We call natural filtration of a process $X=(X_{s,t},(s,t)\in\R_+^2)$ the system $(\mathcal{F}^X_{s,t},(s,t)\in\R_+^2)$ of sub-$\sigma$-algebras of $\mathcal{F}$ given by
		\begin{align*}
			\mathcal{F}_{s,t}^X=\sigma(X_{u,v},0\leq u\leq s,0\leq v\leq t).
		\end{align*}
	\end{defi}
	\begin{defi}
		We call a one-dimensional $(\mathcal{F}_{s,t})$-Brownian sheet on a filtered probability space $(\Omega,\mathcal{F},(\mathcal{F}_{s,t},(s,t)\in \R_+^2),\Pb)$ any real valued two-parameter stochastic process $W=(W_{s,t},(s,t)\in \R_+^2)$ satisfying the following conditions:
		\begin{enumerate}
			\item $W$ is $(\mathcal{F}_{s,t})$-adapted, i.e. $W_{s,t}$ is $\mathcal{F}_{s,t}$-measurable, for every $(s,t)\in \R_+^2$.
			\item Almost every sample function $(s,t)\longmapsto W_{s,t}(\omega)$ of $W$ is continuous on $\R^2_+$.
			\item Almost every sample function of $W$ vanishes on $\partial \R_+^2$.
			\item For every finite rectangle of the type $\Pi=]s,\ps]\times]t,\tp]\subset \R_+^2$, the random variable $$W(\Pi):=W_{\ps,\tp}-W_{s,\tp}-W_{\ps,t}+W_{s,t}$$
			is centered, Gaussian with variance $(\ps-s)(\tp-t)$ and independent of $\mathcal{F}_{s,\infty}\vee\mathcal{F}_{\infty,t}$, where
			\begin{align*}
				\mathcal{F}_{s,\infty}=\sigma\big(\bigcup_{v\in\R_+}\mathcal{F}_{s,v}\big)\,\text{ and }\,\mathcal{F}_{\infty,t}=\sigma\big(\bigcup_{u\in\R_+}\mathcal{F}_{u,t}\big).
			\end{align*}	
		\end{enumerate} 
		We call a $d$-dimensional Brownian sheet any $\R^d$-valued two-parameter process $W=(W^{(1)},\ldots,W^{(d)})$ such that $W^{(i)}$, $i=1,\ldots,d$, are independent one-dimensional Brownian sheets.
	\end{defi}

	One useful caracterisation of the Brownian sheet $W$ is that it is the only one centered and continuous Gaussian process with covariance function $[(s,t),(\ps,\tp)]\longmapsto\E[W_{s,t}W_{\ps,\tp}]=(s\wedge\ps)(t\wedge\tp)$.
	As a consequence, if $W$ is a Brownian sheet on a given filtered probability space, then so are the processes $(W_{a+s,t}-W_{a,t},(s,t)\in\R^2)$ and $\Big(\ve^{-1/2}(W_{s,a+\ve t}-W_{s,a}),(s,t)\in\R_+^2\Big)$ with respect to their natural filtrations. We mention that a rather complete analysis on multi-parameter processes and their applications in Analysis is provided by Khoshnevisan \cite{Kh02}.		
	
	The remainder of paper consists of two sections. In the first one, we propose a stochastic local time-space calculus for the Brownian sheet. We then extend several formulas obtained by Eisenbaum \cite{Ei00} to other classes of Brownian sheet processes. We also generalise an It\^o formula obtained by Cairoli and Walsh \cite{CW75}. The last section is devoted to an application of the local time-space calculus presented in the previous section. More specifically, we give several Davie types inequalities for the Brownian sheet are given and exploited to obtain regularity estimates of some averaging operators .

	\section{ Stochastic Integration over the space with respect to Local Time}
	
	\subsection{Integration with respect to local time of deterministic functions}\label{sectloct1}
	
	We aim at writting the integral of a $\R$-valued Borel measurable function on $[0,1]^2\times\R$ with respect to the local time in the plane of the Brownian sheet.
	Let $(W_{s,t};s\geq0,t\geq0)$ be a Brownian sheet given on an equipped probability space. It is known that for $s$ fixed, $(W_{s,t},t\geq0)$ is a Brownian motion, and its local time process $(L_1^x(s,t);x\in\R,t\geq0)$ is given by Tanaka's formula (see for example \cite[Section 1]{Wa78}): 
	\begin{align}\label{eq:TanakaSheet1}
		\int_0^t\mathbf{1}_{\{W_{s,u}\leq x\}}\mathrm{d}_uW_{s,u}=\frac{s}{2}L^x_1(s,t)-(W_{s,t}-x)^{-}+x^{+}.
	\end{align}
	Moreover for any fixed $s\in[0,1]$, let $\hat{W}_{s,\cdot}$ be the time reversal process on $[0,1]$ of the Brownian motion $\hat{W}_{s,\cdot}$, i.e. $\hat{W}_{s,t}=W_{s,1-t}$, and let $(\hat{L}^x_1(s,t);x\in\R,0\leq t\leq1)$ be the local time process  of $(\hat{W}_{s,t},0\leq t\leq 1)$. Then the following holds
	\begin{align*}
		\hat{L}^x_1(s,t)=L^x_1(s,1)-L^x_1(s,1-t).
	\end{align*}
	It follows from Tanaka's formula that
	\begin{align}\label{eq:TanakaSheet2}
		\int_{1-t}^1\mathbf{1}_{\{\hat{W}_{s,u}\leq x\}}\mathrm{d}_u\hat{W}_{s,u}=\frac{s}{2}L^x_1(s,t)+(\hat{W}_{s,1-t}-x)^{-}-(\hat{W}_{s,1}-x)^{-}.
	\end{align}
	Summing \eqref{eq:TanakaSheet1} and \eqref{eq:TanakaSheet2} yields
	\begin{align}\label{eq:StoInt1DLocTime01}
		sL^x_1(s,t)&=\int_0^t\mathbf{1}_{\{W_{s,u}\leq x\}}\mathrm{d}_uW_{s,u}+\int_{1-t}^1\mathbf{1}_{\{\hat{W}_{s,u}\leq x\}}\mathrm{d}_u\hat{W}_{s,u}.
	\end{align}
	Next we introduce a notion of backward stochastic integral with respect to the Brownian motion $(W_{s,t},t\geq0)$. Let $f:\,[0,1]^2\times\R\to\R$ be a measurable function such that $x\longmapsto f(s,t,x)$ is locally square integrable for every $s,t\in[0,1]$ and $(s,t)\longmapsto f(s,t,x)$ is weakly continuous for every $x\in\R$, that is $(s,t)\longmapsto f(s,t,x)$ is continuous for every $x\in\R$ as a map from $[0,1]^2$ to $L^2_{loc}(\R)$. For $s,t\in[0,1]$ fixed, it was shown in \cite[Proposition 3.2]{FPS95} that the limit below exists in $L^1(\Omega,\Pb)$: 
	\begin{align*}
		\lim\limits_{\sup_k|t_{k+1}-t_k|\to0}\sum\limits_{0<t_1<\cdots<t_{\ell}<t}f(s,t_{k+1},W_{s,t_{k+1}})(W_{s,t_{k+1}}-W_{s,t_k}),
	\end{align*}
	where $(t_k)_{1\leq k\leq \ell}$ is a subdivision of $[0,t]$.
	This limit is denoted by $\int_0^tf(s,u,W_{s,u})\mathrm{d}_u^{\ast}W_{s,u}$ and called a backward stochastic integral with respect to the Brownian motion $(W_{s,t},0\leq t\leq1)$. In addition, using\cite[Eq. (3.19), page 153]{FPS95}) and \cite[Chapter II, Theorem 11]{Pro90}, we also have
	\begin{align}\label{eq:FPSRevBSheet01}
		\int_0^tf(s,u,W_{s,u})\mathrm{d}_u^{\ast}W_{s,u}=-\int_{1-t}^1f(s,1-u,\hat{W}_{s,u})\mathrm{d}_u\hat{W}_{s,u}.
	\end{align}
	It follows from (\ref{eq:StoInt1DLocTime01}) and (\ref{eq:FPSRevBSheet01}) that
	\begin{align}\label{eq:FPSRevBSheet02}
		sL^x_1(s,t)=\int_0^t\mathbf{1}_{\{W_{s,u}\leq x\}}\mathrm{d}_uW_{s,u}-\int_{0}^t\mathbf{1}_{\{W_{s,u}\leq x\}}\mathrm{d}^{\ast}_uW_{s,u}.
	\end{align}
	We call local time for the Brownian sheet $W$ the process $L:=(L_{s,t}^x;x\in\R,s\geq0,t\geq0)$ defined in \cite[Section 6, Page 157]{CW75} (see also \cite[Section 2]{Wa78}) by :
	\begin{align}\label{DefiLocTimeW}
		L^x_{s,t}=\lim\limits_{\ve\to0}\frac{1}{2\ve}\int_0^s\int_0^t\mathbf{1}_{[x-\ve,x+\ve]}(W_{s_1,t_1})\,\mathrm{d}t_1\mathrm{d}s_1.
	\end{align}
	We deduce from \eqref{DefiLocTimeW} (see e.g. \cite[Eq. (2.3)]{Wa78}) that
	\begin{align}\label{DefLocTimeSheet}
		\int_0^sL_1^x(s_1,t)\mathrm{d}s_1=	L^x_{s,t}=\int_0^tL_2^x(s,t_1)\mathrm{d}t_1,\quad\forall\,x\in\R,\,\forall\,(s,t)\in\R_+^2,
	\end{align}
	where $(L^x_2(s,t);x\in\R,s\geq0)$ is the local time of the Brownian motion $(W_{s,t},s\geq0)$.
	
	%

	Substituting \eqref{eq:FPSRevBSheet02} into \eqref{DefLocTimeSheet} yields
	\begin{align}\label{eq:StoInt2DLocTime01}
		L^x_{s,t}=\int_0^s\int_0^t\mathbf{1}_{\{W_{\xi,u}\leq x\}}\frac{\mathrm{d}_uW_{\xi,u}}{\xi}\mathrm{d}\xi-\int_0^s\int_{0}^t\mathbf{1}_{\{W_{\xi,u}\leq x\}}\frac{\mathrm{d}^{\ast}_uW_{\xi,u}}{\xi}\mathrm{d}\xi.
	\end{align}
	
	Denote by $(\mathcal{H}, \Vert\cdot\Vert)$ the space of Borel measurable functions $f:[0,T]\times \mathbb{R}^2\rightarrow \mathbb{R} $ with the norm $\Vert\cdot\Vert$ defined by
	\begin{align*}
		\Vert f\Vert=&2\Big(\int_0^1\int_0^1\int_{\R}f^2(s,t,x)\exp\Big(-\frac{x^2}{2st}\Big)\frac{\mathrm{d}x\mathrm{d}s\mathrm{d}t}{\sqrt{2\pi st}}\Big)^{1/2}\\
		&+\int_0^1\int_0^1\int_{\R}|xf(s,t,x)|\exp\Big(-\frac{x^2}{2st}\Big)\frac{\mathrm{d}x\mathrm{d}s\mathrm{d}t}{st\sqrt{2\pi st}}\\
		=&2\Big(\int_0^1\int_0^1\E\Big[f^2(s,t,W_{s,t})\Big]\mathrm{d}s\mathrm{d}t\Big)^{1/2}+\int_0^1\int_0^1\E\Big[\Big|f(s,t,W_{s,t})\frac{W_{s,t}}{st}\Big|\Big] \mathrm{d}s\mathrm{d}t.
	\end{align*}
	Then, endowed with the above norm, $\mathcal{H}$ is a Banach space. We show(see Proposition \ref{prop:StoInt3DDcase})  that one can define a stochastic integration over the time and space with respect to local time for the elements of $\mathcal{H}$. 
	
	We say that $f_{\Delta}:\,[0,1]^2\times\R\to\R$ is an elementary function if there exist two sequences of real numbers $(x_i)_{0\leq i\leq n}$, $(f_{ijk};0\leq i\leq n,0\leq j\leq m, 0\leq k\leq\ell)$ and two subdivisions of $[0,1]$ $(s_j)_{0\leq j\leq m}$, $(t_k)_{0\leq k\leq \ell}$ such that
	\begin{align}\label{eq:ElemFunction}
		f_{\Delta}(s,t,x)=\sum\limits_{(x_i,s_j,t_k)\in\Delta}f_{ijk}\mathbf{1}_{(x_i,x_{i+1}]}(x)\mathbf{1}_{(s_j,s_{j+1}]}(s)\mathbf{1}_{(t_k,t_{k+1}]}(t),
	\end{align}
	where $\Delta=\{(x_i,s_j,t_k);0\leq i\leq n,0\leq j\leq m, 0\leq k\leq\ell\}$.
	For such a given function $f_{\Delta}$, we define its integral with respect to $L$ as 
	\begin{align*}
		\int_0^1\int_0^1\int_{\R}f_{\Delta}(s,t,x)\mathrm{d}L^x_{s,t}=&\sum\limits_{(x_i,s_j,t_k)\in\Delta}f_{ijk}\Big(L^{x_{i+1}}_{s_{j+1},t_{k+1}}-L^{x_{i+1}}_{s_{j},t_{k+1}}-L^{x_{i}}_{s_{j+1},t_{k+1}}+L^{x_{i}}_{s_{j},t_{k+1}}\\
		&	-L^{x_{i+1}}_{s_{j+1},t_{k}}+L^{x_{i+1}}_{s_{j},t_{k}}+L^{x_{i}}_{s_{j+1},t_{k}}-L^{x_{i}}_{s_{j},t_{k}}\Big).
	\end{align*} 
	Let $f$ be an element of $\mathcal{H}$ and let $(f_n)_{n\in\N}$ be a sequence of elementary functions converging to $f$ in $\mathcal{H}$. We prove in the following result that $\Big(\int_0^1\int_0^1\int_{\R}f_n(s,t,x)\mathrm{d}L^x_{s,t}\Big)_{n\in\N}$ converge in $L^1(\Omega,\Pb)$ and that the limit does not depend of the choice of the sequence $(f_n)_{n\in\N}$. This limit is called integral of $f$ with respect to $L$. 
	\begin{prop}\label{prop:StoInt3DDcase}
		For any $f\in\mathcal{H}$, the integral
		$\int_0^s\int_0^t\int_{\R}f(\xi,u,x)\mathrm{d}L^x_{\xi,u}$ exists and is given for any $(s,t)\in(0,1]^2$ by
		\begin{align}\label{eq:StoInt3DLocTime01}
			\int_0^s\int_0^t\int_{\R}f(\xi,u,x)\mathrm{d}L^x_{\xi,u}=\int_0^s\int_0^tf(\xi,u,W_{\xi,u})\frac{\mathrm{d}_uW_{\xi,u}}{\xi}\mathrm{d}\xi-\int_0^s\int_0^tf(\xi,u,W_{\xi,u})\frac{\mathrm{d}^{\ast}_uW_{\xi,u}}{\xi}\mathrm{d}\xi.
		\end{align}
		Moreover, we have
		\begin{align}\label{Ineq:StoInt3D02}
			\E\Big[\Big|\int_0^s\int_0^t\int_{\R}f(\xi,u,x)\mathrm{d}L^x_{\xi,u}\Big|\Big]\leq\Vert f\Vert.
		\end{align}
	\end{prop}
	\begin{proof}
		We first show the Proposition for elementary function. Let $f_{\Delta}$ be a simple function as defined in \eqref{eq:ElemFunction}. We deduce from \eqref{eq:StoInt2DLocTime01} that
		\begin{align*}
			&L^{x_{i+1}}_{s_{j+1},t_{k+1}}-L^{x_{i+1}}_{s_{j},t_{k+1}}-L^{x_{i}}_{s_{j+1},t_{k+1}}+L^{x_{i}}_{s_{j},t_{k+1}}-L^{x_{i+1}}_{s_{j+1},t_{k}}+L^{x_{i+1}}_{s_{j},t_{k}}+L^{x_{i}}_{s_{j+1},t_{k}}-L^{x_{i}}_{s_{j},t_{k}}\\
			=&\int_{s_j}^{s_{j+1}}\int_{t_k}^{t_{k+1}}\mathbf{1}_{\{x_i< W_{\xi,u}\leq x_{j+1}\}}\frac{\mathrm{d}_uW_{\xi,u}}{\xi}\mathrm{d}\xi-\int_{s_j}^{s_{j+1}}\int_{t_k}^{t_{k+1}}\mathbf{1}_{\{x_i< W_{\xi,u}\leq x_{j+1}\}}\frac{\mathrm{d}^{\ast}_uW_{\xi,u}}{\xi}\mathrm{d}\xi\\
			=&\int_0^1\int_0^1\mathbf{1}_{(s_j,s_{j+1}]}(\xi)\mathbf{1}_{(t_k,t_{k+1}]}(u)\mathbf{1}_{]x_i,x_{i+1}]}(W_{\xi,u})\frac{\mathrm{d}_uW_{\xi,u}}{\xi}\mathrm{d}\xi
			\\&-\int_0^1\int_0^1\mathbf{1}_{(s_j,s_{j+1}]}(\xi)\mathbf{1}_{(t_k,t_{k+1}]}(u)\mathbf{1}_{]x_i,x_{i+1}]}(W_{\xi,u})\frac{\mathrm{d}^{\ast}_uW_{\xi,u}}{\xi}\mathrm{d}\xi.
		\end{align*}
		As a consequence, we get
		\begin{align}
			&\int_0^s\int_0^t\int_{\R}f_{\Delta}(\xi,u,x)\mathrm{d}L^x_{\xi,u} \nonumber\\
			=&\int_0^s\int_0^tf_{\Delta}(\xi,u,W_{\xi,u})\frac{\mathrm{d}_uW_{\xi,u}}{\xi}\mathrm{d}\xi-\int_0^s\int_0^tf_{\Delta}(\xi,u,W_{\xi,u})\frac{\mathrm{d}^{\ast}_uW_{\xi,u}}{\xi}\mathrm{d}\xi.\label{eq:IntLocTimeDelta}
		\end{align}
		Let us now show \eqref{Ineq:StoInt3D02} holds for such a simple function. Observe that 
		\begin{align}\label{eqrevB1}
			\int_0^s\int_0^tf_{\Delta}(\xi,u,W_{\xi,u})\frac{\mathrm{d}^{\ast}_uW_{\xi,u}}{\xi}\mathrm{d}\xi=-\int_0^s\int_{1-t}^1f_{\Delta}(\xi,1-u,\hat{W}_{\xi,u})\frac{\mathrm{d}_u\hat{W}_{\xi,u}}{\xi}\mathrm{d}\xi.
		\end{align}
		We also know from \cite[Theorem 6.1]{DaW02} that $\hat{W}$ admits the following representation
		\begin{align}\label{eq:ReprDWTimeRev}
			\hat{W}_{s,t}=W_{s,1}+B_{s,t}-\int_0^t\frac{\hat{W}_{s,u}}{1-u}\mathrm{d}u,
		\end{align}
		where $B$ is a standard Brownian sheet independent of $(W_{s,1},s\geq0)$. Using \eqref{eq:FPSRevBSheet01} and \eqref{eqrevB1}, \eqref{eq:IntLocTimeDelta} can be rewritten as
		\begin{align}
			&\int_0^s\int_0^tf_{\Delta}(\xi,u,W_{\xi,u})\frac{\mathrm{d}^{\ast}_uW_{\xi,u}}{\xi}\mathrm{d}\xi\nonumber\\
			=&-\int_0^s\int_{1-t}^1f_{\Delta}(\xi,1-u,\hat{W}_{\xi,u})\frac{\mathrm{d}_uB_{\xi,u}}{\xi}\mathrm{d}\xi+\int_0^s\int_{1-t}^1f_{\Delta}(\xi,1-u,\hat{W}_{\xi,u})\frac{\hat{W}_{\xi,u}}{\xi(1-u)}\mathrm{d}u\,\mathrm{d}\xi.\label{eq:IntLocTimeDeltaR}
		\end{align}
		Substitute \eqref{eq:IntLocTimeDeltaR} into \eqref{eq:IntLocTimeDelta}, take the absolute value on both sides, use the triangle, take the expectation and use the Cauchy-Schwarz inequalities to obtain
		\begin{align*}
			&\E\Big[\Big|\int_0^s\int_0^t\int_{\R}f_{\Delta}(\xi,u,x)\mathrm{d}L^x_{\xi,u}\Big|\Big]\\
			\leq&\E\Big[\Big|\int_0^s\int_0^tf_{\Delta}(\xi,u,W_{\xi,u})\frac{\mathrm{d}_uW_{\xi,u}}{\xi}\mathrm{d}\xi\Big|\Big]+\E\Big[\Big|\int_0^s\int_0^tf_{\Delta}(\xi,u,W_{\xi,u})\frac{\mathrm{d}^{\ast}_uW_{\xi,u}}{\xi}\mathrm{d}\xi\Big|\Big]\\
			\leq	&\Big(\int_0^1\int_0^1\E\Big[f^2_{\Delta}(\xi,u,W_{\xi,u})\Big]\mathrm{d}u\,\mathrm{d}\xi\Big)^{1/2}+\Big(\int_0^1\int_{0}^1\E\Big[f^2_{\Delta}(\xi,1-u,\hat{W}_{\xi,u})\Big]\mathrm{d}u\,\mathrm{d}\xi\Big)^{1/2}\\
			&+\int_0^1\int_{0}^1\E\Big[\Big|f_{\Delta}(\xi,1-u,\hat{W}_{\xi,u})\frac{\hat{W}_{\xi,u}}{\xi(1-u)}\Big|\Big]\mathrm{d}u\,\mathrm{d}\xi,
		\end{align*}
		which means that
		\begin{align}\label{Ineq:StoInt3D01}
			\E\Big[\Big|\int_0^s\int_0^t\int_{\R}f_{\Delta}(\xi,u,x)\mathrm{d}L^x_{\xi,u}\Big|\Big]\leq\Vert f_{\Delta}\Vert.
		\end{align}
		Thus \eqref{Ineq:StoInt3D02} holds for simple functions. Using the above inequality, the extension of the definition of the integral to the elements of $\mathcal{H}$  follows by the density of elementary functions in $\mathcal{H}$. Since the integrals 
		$$
		\int_0^s\int_0^tf(\xi,u,W_{\xi,u})\frac{\mathrm{d}_uW_{\xi,u}}{\xi}\mathrm{d}\xi\,\text{ and }\,\int_0^s\int_0^tf(\xi,u,W_{\xi,u})\frac{\mathrm{d}^{\ast}_uW_{\xi,u}}{\xi}\mathrm{d}\xi
		$$ 
		exist and are well defined, it follows that any function $f$ in $\mathcal{H}$ satisfies:
		\begin{align*}
			\int_0^s\int_0^t\int_{\R}f(\xi,u,x)\mathrm{d}L^x_{\xi,u}=\int_0^s\int_0^tf(\xi,u,W_{\xi,u})\frac{\mathrm{d}_uW_{\xi,u}}{\xi}\mathrm{d}\xi-\int_0^s\int_0^tf(\xi,u,W_{\xi,u})\frac{\mathrm{d}^{\ast}_uW_{\xi,u}}{\xi}\mathrm{d}\xi.
		\end{align*}
		This gives \eqref{eq:StoInt3DLocTime01}.
		To obtain \eqref{Ineq:StoInt3D02}, we apply \eqref{Ineq:StoInt3D01} and a limit argument.
	\end{proof}
	\begin{remark}\label{rem:StoIntLTContF}
		\item[1.] For any $0\leq s_1<s_2$, $0\leq t_1<t_2$ and $x_1<x_2$, we set $A=[s_1,s_2]\times[t_1,t_2]\times[x_1,x_2]$. Then one has
		\begin{align*}
			\int_A\mathrm{d}L^x_{s,t}=&	L^{x_{2}}_{s_{2},t_{2}}-L^{x_{2}}_{s_{1},t_{2}}-L^{x_{1}}_{s_{2},t_{2}}+L^{x_{1}}_{s_{1},t_{2}}
			-L^{x_{2}}_{s_{2},t_{1}}+L^{x_{2}}_{s_{1},t_{1}}+L^{x_{1}}_{s_{2},t_{1}}-L^{x_{1}}_{s_{1},t_{1}}\\
			=&\int_{s_1}^{s_2}\left\{L_1^{x_2}(\xi,t_2)-L_1^{x_1}(\xi,t_2)-L_1^{x_2}(\xi,t_1)+L_1^{x_1}(\xi,t_1)\right\}\mathrm{d}\xi=\int_Ad_{x,t}L^{x}_1(\xi,t)\,\mathrm{d}\xi.
		\end{align*}
		By a monotone class argument, 
		\begin{align}\label{MonotoneClass}
			\int_0^1\int_0^1\int_{\R}g(s,t,x)\mathrm{d}L^{x}_{s,t}=\int_0^1\int_0^1\int_{\R}g(s,t,x)\mathrm{d}_{x,t}L^{x}_1(\xi,t)\,\mathrm{d}\xi
		\end{align}
		for every bounded Borel measurable function $g:\,[0,1]^2\times\R\to\R$. In particular, for any bounded Borel measurable function $\ell:\,[0,1]^2\to\R$ and any $a\in\R$,
		\begin{align}\label{MonotoneClass2}
			\int_0^1\int_0^1\ell(s,t)\mathrm{d}_{s,t}L^{a}_{s,t}=\int_0^1\int_0^1\ell(s,t)\mathrm{d}_{t}L^{a}_1(\xi,t)\,d\xi.
		\end{align}
		\item[2.]	Let $f:\,[0,1]^2\times\R\to\R$ be a continuous function in $\mathcal{H}$. For $a<b$, let $(x_i)_{0\leq i\leq n}$ be a subdivision of $[a,b]$, $(s_j)_{0\leq j\leq m}$ be a subdivision of $[0,s]$ and $(t_k)_{0\leq k\leq\ell}$ be a subdivision of $[0,t]$. Denote by $\Delta$ the grid $\{(s_j,t_k,x_i),0\leq i\leq n,0\leq j\leq m, 0\leq k\leq\ell\}$. Then, as $|\Delta|$ tends to $0$, the expression 
		\begin{align*} 
			\sum\limits_{
				\begin{subarray}{}
					0\leq i\leq n,0\leq j\leq m\\
					\quad 0\leq k\leq\ell
				\end{subarray}
			}f(s_j,t_k,x_i)\left(L^{x_{i+1}}_{s_{j+1},t_{k+1}}-L^{x_{i+1}}_{s_{j},t_{k+1}}-L^{x_{i}}_{s_{j+1},t_{k+1}}+L^{x_{i}}_{s_{j},t_{k+1}}\right.\\
			\left.-L^{x_{i+1}}_{s_{j+1},t_{k}}+L^{x_{i+1}}_{s_{j},t_{k}}+L^{x_{i}}_{s_{j+1},t_{k}}-L^{x_{i}}_{s_{j},t_{k}}\right)
		\end{align*}
		converges in $L^1$ to $\int_0^t\int_0^s\int_a^bf(s,t,x)\,\mathrm{d}L^x_{s,t}$. In particular, when $f$ is differentiable with respect to $x$ and $\partial_xf$ is continuous on $[0,1]^2\times\R$, we deduce from \cite[Theorem 5.1 (i) ]{Ei00} that for any $s\in[0,1]$,
		\begin{align*}
			&\int_0^t\int_a^bf(s,u,x)\,d_{x,u}L^x_1(s,u)
			\\&=-\int_0^t(\mathbf{1}_{[a,b]}\partial_xf)(s,u,W_{s,u})\mathrm{d}u+\int_0^tf(s,u,b)\mathrm{d}_{u}L^b_1(s,u)-\int_0^tf(s,u,a)\mathrm{d}_{u}L^a_1(s,u).
		\end{align*}
		Hence, integrating over $[0,s]$, using \eqref{MonotoneClass} and \eqref{MonotoneClass2} and integrating over $[0,s]$ give
		\begin{align}\label{eqLocTimeIntDet}
			\int_0^t\int_0^s\int_a^bf(s,t,x)\,\mathrm{d}L^x_{s,t}&=-\int_0^s\int_0^t(\mathbf{1}_{[a,b]}\partial_xf)(\xi,u,W_{\xi,u})\mathrm{d}u\,\mathrm{d}\xi+\int_0^s\int_0^tf(\xi,u,b)\mathrm{d}_{\xi,u}L^b_{\xi,u}\nonumber\\&\qquad-\int_0^s\int_0^tf(\xi,u,a)\mathrm{d}_{\xi,u}L^a_{\xi,u}.
		\end{align}
		Hence, letting $a$ (respectively $b$) goes to $-\infty$ (respectively $\infty$), we obtain
		\begin{align}\label{eqStoIntLTContF}
			\int_0^t\int_0^s\int_{\R}f(s,t,x)\,\mathrm{d}L^x_{s,t}=-\int_0^s\int_0^t\partial_xf(\xi,u,W_{\xi,u})\mathrm{d}u\,\mathrm{d}\xi.
		\end{align}
	\end{remark}
	\begin{corollary}\label{corol:EisenSheetdD}
		Let $\Big(W_{s,t}:=(W_{s,t}^{(1)},\cdots,W_{s,t}^{(d)});s\geq0,t\geq0\Big)$ be a $d$-dimensional Brownian sheet defined on an equipped probability space. Let $f:\,[0,1]^2\times\R^d\to\R$ be a continuous function such that for any $(s,t)\in[0,1]^2$, $f(s,t,\cdot)$ is differentiable and for any $i\in\{1,\ldots,d\}$, the partial derivative $\partial_{x_i}f$ is continuous. Then for any $(s,t)\in[0,1]^2$ and any $i\in\{1,\ldots,d\}$, we have
		\begin{align}\label{eq:EisenSheetdD01}
			&\int_0^s\int_0^t\partial_{x_i}f(\xi,u,W_{\xi,u})\mathrm{d}u\,\mathrm{d}\xi=-\int_0^s\int_0^tf(\xi,u,W_{\xi,u})\frac{\mathrm{d}_uW^{(i)}_{\xi,u}}{\xi}\mathrm{d}\xi\nonumber\\
			&\qquad-\int_0^s\int_{1-t}^1f(\xi,1-u,\hat{W}_{\xi,u})\frac{\mathrm{d}_uB^{(i)}_{\xi,u}}{\xi}\mathrm{d}\xi+\int_0^s\int_{1-t}^1f(\xi,1-u,\hat{W}_{\xi,u})\frac{\hat{W}^{(i)}_{\xi,u}}{\xi(1-u)}\mathrm{d}u\,\mathrm{d}\xi,
		\end{align}
		where $\hat{W}^{(i)}_{\xi,u}=W^{(i)}_{\xi,1-u}$ and $B^{(i)}$ is a standard Brownian sheet independent of $(W^{(i)}_{s,1},s\geq0)$. Consequently for any $s\in]0,1]$, the $d$-dimensional Brownian motion $(W_{s,t},t\geq0)$ satisfies
		\begin{align}\label{eq:EisenSheetdD02}
			&\int_0^t\partial_{x_i}f(s,u,W_{s,u})\mathrm{d}u=-\int_0^tf(s,u,W_{s,u})\frac{\mathrm{d}_uW^{(i)}_{s,u}}{s}\nonumber\\
			&
			\qquad-\int_{1-t}^1f(s,1-u,\hat{W}_{s,u})\frac{\mathrm{d}_uB^{(i)}_{s,u}}{s}+\int_{1-t}^1f(s,1-u,\hat{W}_{s,u})\frac{\hat{W}^{(i)}_{s,u}}{s(1-u)}\mathrm{d}u.
		\end{align}
	\end{corollary}
	
	\begin{proof} 
		The proof is analogeous to \cite[Section 6]{Ei06}. We denote by $\left(L^x_{s,t}(W^{(i)});x\in\R,s\geq0,t\geq0\right)$ the local time on the plane of $W^{(i)}$ and we adopt the notation
		\begin{align*}
			g(s,t,W_{s,t}^{(1)},\cdots,W_{s,t}^{(i-1)},x,W_{s,t}^{(i+1)},\cdots,W_{s,t}^{(d)})=g(s,t,W_{s,t})\vert_{W_{s,t}^{(k)}=x}.
		\end{align*}
		For any measurable function $g:\,[0,1]^2\times\R^d\to\R$, we define the norm $\Vert\cdot\Vert_i$ by
		\begin{align*}
			\Vert g\Vert_i=2\Big(\int_0^1\int_0^1\E\left[g^2(s,t,W_{s,t})\right]\mathrm{d}s\mathrm{d}t\Big)^{1/2}+\int_0^1\int_0^1\E\Big[\Big|g(s,t,W_{s,t})\frac{W^{(i)}_{s,t}}{st}\Big|\Big] \mathrm{d}s\mathrm{d}t.
		\end{align*}
		For any continuous function $f:\,[0,1]^2\times\R^d\to\R$ and any $i\in\{1,\ldots,d\}$, we note that, conditionally to $(W^{(k)}_{s,t},(s,t)\in[0,1]^2)_{1\leq k\leq d,k\neq i}$, $f(s,t,W_{s,t},0\leq s,t\leq 1)$ is a deterministic function of $(W^{(i)}_{s,t},0\leq s,t\leq 1)$. Suppose that $f(s,t,\cdot)$ is differentiable for any $s,t$, its partial derivative $\partial_{x_i}f$ is continuous and $\Vert f\Vert_i<\infty$ for any $i$. Then, using \eqref{eqStoIntLTContF} in Remark \ref{rem:StoIntLTContF}, relation  \cite[(3.19) ]{FPS95}, Proposition \ref{prop:StoInt3DDcase}  and \cite[Chapter II, Theorem 11]{Pro90} we have
		\begin{align}
			&\int_0^s\int_0^t\partial_{x_i}f(\xi,u,W_{\xi,u})\mathrm{d}u\,\mathrm{d}\xi=-\int_0^s\int_0^t\int_{\R}f(s,t,W_{s,t})\vert_{W_{s,t}^{(i)}=x}\,\mathrm{d}L^x_{s,t}(W_{s,t}^{(i)})\nonumber\\
			=&-\int_0^s\int_0^tf(\xi,u,W_{\xi,u})\frac{\mathrm{d}_uW^{(i)}_{\xi,u}}{\xi}\mathrm{d}\xi+\int_0^s\int_0^tf(\xi,u,W_{\xi,u})\frac{\mathrm{d}^{\ast}_uW^{(i)}_{\xi,u}}{\xi}\mathrm{d}\xi\nonumber\\
			=&-\int_0^s\int_0^tf(\xi,u,W_{\xi,u})\frac{\mathrm{d}_uW^{(i)}_{\xi,u}}{\xi}\mathrm{d}\xi-\int_0^s\int_{1-t}^1f(\xi,1-u,\hat{W}_{\xi,u})\frac{\mathrm{d}_u\hat{W}^{(i)}_{\xi,u}}{\xi}\mathrm{d}\xi.\label{eq:LocTimeIntMultiD}
		\end{align}
		Equation \eqref{eq:EisenSheetdD01} is obtained by substituting \eqref{eq:ReprDWTimeRev} into \eqref{eq:LocTimeIntMultiD}.  We derive \eqref{eq:EisenSheetdD02} by differentiating both sides of \eqref{eq:EisenSheetdD01} with respect to $s$.
	\end{proof}	
	
	\subsection{Local time-space  integration of two parameter random processes and a generalized It\^o formula}
	In this section, we derive an It\^o formula for two parameter random processes. 
	
	Let $h:\,[0,1]^2\times\Omega\times\R\to\R$ be a random function. As before for $a<b$, consider $(x_i)_{0\leq i\leq n}$ a subdivision of $[a,b]$, $(s_j)_{0\leq j\leq m}$ a subdivision of $[0,s]$ and $(t_k)_{0\leq k\leq\ell}$ a subdivision of $[0,t]$. Denote by $\Delta$ the grid $\{(s_j,t_k,x_i),0\leq i\leq n,0\leq j\leq m, 0\leq k\leq\ell\}$. When $h$ is regular enough, we show that, as $|\Delta|$ tends to $0$, the expression 
	\begin{align*}
		&	\sum\limits_{
			\begin{subarray}{}
				0\leq i\leq n,0\leq j\leq m\\
				\quad 0\leq k\leq\ell
			\end{subarray}
		}h(s_j,t_k,\omega,x_i)\Big(L^{x_{i+1}}_{s_{j+1},t_{k+1}}-L^{x_{i+1}}_{s_{j},t_{k+1}}-L^{x_{i}}_{s_{j+1},t_{k+1}}+L^{x_{i}}_{s_{j},t_{k+1}}\\
		&\quad\quad\quad-L^{x_{i+1}}_{s_{j+1},t_{k}}+L^{x_{i+1}}_{s_{j},t_{k}}+L^{x_{i}}_{s_{j+1},t_{k}}-L^{x_{i}}_{s_{j},t_{k}}\Big)
	\end{align*}
	admits a limit in $L^1$ for any $\omega\in\Omega$, denoted by $\int_0^t\int_0^s\int_a^bh(s,t,\omega,x)\,dL^x_{s,t}$. 
	
	The next result extends \eqref{eqStoIntLTContF} to random functions and is a direct consequence of \cite[Theorem 5.1]{Ei00} for local times of the Brownian sheet.  
	\begin{prop}\label{prop:StoIntLTContF}
		Let $h:\,[0,1]^2\times\Omega\times\R\to\R$ be a random real valued function such that for any $\omega\in\Omega$, $(s,t,x)\longmapsto h(s,t,\omega,x)$ is continuous on $[0,1]^2\times\R$, and for any $(s,t,\omega)\in[0,1]^2\times\Omega$, $h(s,t,\omega,\cdot)$ is differentiable. We suppose that the partial derivative $\partial_xh$ is continuous on $[0,1]^2\times\R$ for $\Pb$-a.e. $\omega\in\Omega$. Then for any $(s,t)\in[0,1]^2$, any $(a,b)\in\R^2$ with $a<b$, and for $\Pb$-a.e. $\omega\in\Omega$, the integral $\int_0^s\int_0^t\int_{\R}h(\xi,u,\omega,x)\,\mathrm{d}L^x_{\xi,u}$ exists and we have
		\begin{align*}
			\int_0^s\int_0^t\int_a^bh(\xi,u,\omega,x)\mathrm{d}L^x_{\xi,u}
			=&-\int_0^s\int_0^t(\mathbf{1}_{[a,b]}\partial_xh)(\xi,u,\omega,W_{\xi,u})\mathrm{d}u\mathrm{d}\xi\\
			&+\int_0^s\int_0^th(\xi,u,\omega,b)\mathrm{d}_{\xi,u}L^b_{\xi,u}-\int_0^s\int_0^th(\xi,u,\omega,a)\mathrm{d}_{\xi,u}L^a_{\xi,u}.
		\end{align*}
		Consequently, when $a$ and $b$ tend respectively to $-\infty$ and $\infty$, we get
		\begin{align}\label{eqLoctimeSheet}
			\int_0^s\int_0^t\int_{\R}h(\xi,u,\omega,x)\mathrm{d}L^x_{\xi,u}=-\int_0^s\int_0^t\partial_xh(\xi,u,\omega,W_{\xi,u})\mathrm{d}u\mathrm{d}\xi.
		\end{align}
	\end{prop}
	
	The relation \eqref{eqLoctimeSheet} provides the definition of the integral  with respect to the local time $L$ for a smooth random function $h$.
	
	Let $(J_{s,t},(s,t)\in D)$ be the process defined by
	\begin{align*}
		J_{s,t}=\int_0^t\int_0^s\int_0^t\int_0^s\mathbf{1}_{\{u<v,\xi>\zeta\}}\mathrm{d}W_{u,\xi}\mathrm{d}W_{v,\zeta},
	\end{align*}
	where $(W_{s,t};(s,t)\in D)$ denotes a real valued Brownian sheet given on an equipped probability space $(\Omega,\mathcal{F},\{\mathcal{F}_{s,t};(s,t)\in D\},\Pb)$. The next result is a generalised It\^o formula for two-parameter Brownian motion.
	\begin{prop}
		Let $h:\,[0,1]^2\times\Omega\times\R\to\R$ be a random bounded function such that, for any $(s,t,x)$, $h(s,t,\cdot,x)$ is  $\mathcal{F}_{s,t}$-measurable, and, $\Pb$-a.e., the partial derivatives $\hpo$, $\hps$, $\hpp$, $\hsp$ and $\htp$ exist and are continuous. Then
		\begin{align*}
			&	h(s,t,\omega,W_{s,t})-h(0,t,\omega,0)\\
			=&\int_0^s\hps(u,t,\omega,W_{u,t})\mathrm{d}s-\frac{s}{2}\int_0^t\int_{\R}\hpo(s,\xi,,\omega,x)\mathrm{d}_{x,\xi}L_2^x(u,\xi) -\frac{t}{2}\int_0^s\int_{\R}\hpo(u,t,\omega,x)\mathrm{d}_{x,u}L_1^x(u,t)\\
			&+\int_0^t\int_0^s\hpo(u,\xi,\omega,W_{u,\xi})\mathrm{d}W_{u,\xi}+\int_0^t\int_0^s\hpp(u,\xi,\omega,W_{u,\xi})\mathrm{d}J_{u,\xi}\\
			&+\frac{1}{2}\int_0^t\int_0^s\int_{\R}\Big\{u\hsp+\hpo+u\xi\htp\Big\}(u,\xi,\omega,x)\mathrm{d}L_{u,\xi}^x.	 
		\end{align*}
	\end{prop}
	\begin{proof}
		Let $p:\,\R\to\R_+$ be an infinitely differentiable function with compact support such that $\int_{\R}p(y)dy=1$. Let $(h_n,n\in\N)$ be the sequence of random functions defined by
		\begin{align*}
			h_n(s,t,\omega,x)=\int_{\R}h(s,t,\omega,x-\frac{y}{n+1})p(y)\mathrm{d}y.
		\end{align*}
		From Theorem 5.1 (ii) and \cite[Theorem 5.3 ]{Ei00} applied along the line $\xi=$constant, we have
		\begin{align*}
			&\frac{s}{2}\hnt(s,\xi,\omega,W_{s,\xi})=\frac{1}{2}\int_0^s\Big\{u\hnn(u,\xi,\omega,W_{u,\xi})+\hnt(u,\xi,\omega,W_{u,\xi})\Big\}\mathrm{d}\xi\\&\qquad+\int_0^s\frac{u}{2}\htt(u,\xi,\omega,W_{u,\xi})\mathrm{d}_uW_{u,\xi} -\int_0^s\int_{\R}\frac{u\xi}{2}\htt(u,\xi,\omega,x)\mathrm{d}_{x,u}L_1^x(u,\xi)\nonumber\\
			&\qquad=\int_0^s\frac{u}{2}\htt(u,\xi,\omega,W_{u,\xi})\mathrm{d}_uW_{u,\xi}-\frac{1}{2}\int_0^s\int_{\R}\Big\{u\hsn+\hno+u\xi\htt\Big\}(u,\xi,\omega,x)\mathrm{d}_{x,u}L_1^x(u,\xi).
		\end{align*}
		Integrating the above equality over $[0,t]$, we obtain
		\begin{align}\label{ItoFormHn}
			&-\frac{s}{2}\int_0^t\int_{\R}\hno(u,\xi,\omega,x)\mathrm{d}_{x,\xi}L_2^x(u,\xi)=\frac{s}{2}\int_0^t\hnt(u,\xi,\omega,W_{u,\xi})\mathrm{d}\xi\\
			=&\int_0^t\Big\{\int_0^s\frac{u}{2}\htt(u,\xi,\omega,W_{u,\xi})d_uW_{u,\xi}\Big\}\mathrm{d}\xi\notag\\&-\frac{1}{2}\int_0^t\int_0^s\int_{\R}\Big\{u\hsn+\hno+u\xi\htt\Big\}(u,\xi,\omega,x)d_{x,u}L_1^x(u,\xi)\mathrm{d}\xi\nonumber\\
			=&\int_0^t\Big\{\int_0^s\frac{u}{2}\htt(u,\xi,\omega,W_{u,\xi})d_uW_{u,\xi}\Big\}\mathrm{d}\xi-\frac{1}{2}\int_0^t\int_0^s\int_{\R}\Big\{u\hsn+\hno+u\xi\htt\Big\}(u,\xi,\omega,x)\mathrm{d}L_{u,\xi}^x.\nonumber
		\end{align}
		Hence letting $n$ tends to infinity, the dominated convergence theorem yields
		\begin{align}\label{ItoFormHinfty}
			&-\frac{s}{2}\int_0^t\int_{\R}\hpo(u,\xi,\omega,x)\mathrm{d}_{x,\xi}L_2^x(u,\xi)\\ =&\int_0^t\Big\{\int_0^s\frac{u}{2}\htp(u,\xi,\omega,W_{u,\xi})\mathrm{d}_uW_{u,\xi}\Big\}\mathrm{d}\xi-\frac{1}{2}\int_0^t\int_0^s\int_{\R}\Big\{u\hsp+\hpo+u\xi\htp\Big\}(u,\xi,\omega,x)\mathrm{d}L_{u,\xi}^x.\nonumber
		\end{align}
		Moreover, applying \cite[Theorem 5.3 ]{Ei00} along the line $t=$constant, it holds that
		\begin{align}\label{ItoFormConst}
			h(s,t,\omega,W_{s,t})&=h(0,t,\omega,0)+\int_0^s\hps(u,t,\omega,W_{u,t})\mathrm{d}s+\int_0^s\hpo(u,t,\omega,W_{u,t})\mathrm{d}_uW_{u,t}\nonumber\\&\quad-\frac{t}{2}\int_0^s\int_{\R}\hpo(u,t,\omega,x)\mathrm{d}_{x,u}L_1^x(u,t).
		\end{align}	
		Using the Green formula (see for example \cite[Theorem 6.3]{CW75}), we have
		\begin{align*}
			\int_0^s\hpo(u,t,\omega,W_{u,t})d_uW_{u,t}=&\int_0^t\int_0^s\hpo(u,\xi,\omega,W_{u,\xi})\mathrm{d}W_{u,\xi}+\int_0^t\int_0^s\hpp(u,\xi,\omega,W_{u,\xi})\mathrm{d}J_{u,\xi}\nonumber\\
			&+\int_0^t\Big\{\int_0^s\frac{u}{2}\htp(u,\xi,\omega,W_{u,\xi})d_uW_{u,\xi}\Big\}\mathrm{d}\xi.
		\end{align*}	
		Substituting \eqref{ItoFormHinfty} into the above equality, we obtain
		\begin{align}\label{eqGreen2}
			&	\int_0^s\hpo(u,t,\omega,W_{u,t})\mathrm{d}_uW_{u,t}\\ =&\int_0^t\int_0^s\hpo(u,\xi,\omega,W_{u,\xi})\mathrm{d}W_{u,\xi}+\int_0^t\int_0^s\hpp(u,\xi,\omega,W_{u,\xi})\mathrm{d}J_{u,\xi}\nonumber\\
			&-\frac{s}{2}\int_0^t\int_{\R}\hpo(s,\xi,\omega,x)\mathrm{d}_{x,\xi}L_2^x(u,\xi)+\frac{1}{2}\int_0^t\int_0^s\int_{\R}\Big\{u\hsp+\hpo+u\xi\htp\Big\}(u,\xi,\omega,x)\mathrm{d}L_{u,\xi}^x.\nonumber
		\end{align}
		Finally, substituting \eqref{eqGreen2} into \eqref{ItoFormConst} yields
		the desired formula.
	\end{proof}

	\section{Regularising properties of Brownian sheet paths}	
	In this section, we use results in Section \ref{sectloct1} to show regularity properties of some averaging type operator. We first show some bounds.
	\subsection{Davie type Inequalities for the Brownian Sheet}			
	The following estimate will be be used extensively.  
	\begin{prop}\label{prop:DavieSheet1dd}
		Let $W:=\Big(W^{(1)}_{s,t},\cdots,W^{(d)}_{s,t};(s,t)\in[0,1]^2\Big)$ be a $\R^d$-valued Brownian sheet ($d\geq1$) defined on an equipped probability space $(\Omega,\mathcal{F},\mathbb{F},\Pb)$, where $\mathbb{F}=(\mathcal{F}_{s,t};s,t\in[0,1])$. Let   $g\in\mathcal{C}^0\left([0,1]^2,\mathcal{C}^1(\R^{2d})\right)$ such that there exists $\kappa>0$ satisfying
		\begin{equation}\label{eq:BoundSheetdd}
			g(s,t,x,y)\leq \kappa|y|,\quad\text{for all }(s,t,x,y)\in[0,1]^2\times\R^{2d}.
		\end{equation}
		Let $(a,\ve)\in[0,1)\times(0,1)$ such that $a+\ve \leq1$. Then there exist positive constants $\alpha$ and $C$ such that for all $(s,\ps,y)\in]0,1]^2\times\R^d$ with $s\leq\ps$, $(\ps-s,y)\neq(0,0,\ldots,0)$ and all $i\in\{1,\cdots,d\}$, we have
		\begin{align}\label{eq:DavieSheet1dd}
			\E\Big[\exp\Big(\frac{\alpha \sqrt{\ve s}}{|y|+\sqrt{\ps-s}} \Big|\int_0^1\partial_{x_i} g\Big(s,t,\widetilde{W}^{\ve}_{s,t},y+W^{\ve}_{\ps,t}-W^{\ve}_{s,t}\Big)\mathrm{d}t\Big|\Big)\Big]\leq C,
		\end{align} 
		where $\partial_{x_i} g$ denotes the partial derivative of $g$ with respect to one component of the third variable, $|\cdot|$ is the maximum norm on $\R^d$, and $W^{\ve}:=\Big(W^{(\ve,1)}_{s,t},\cdots,W^{(\ve,d)}_{s,t};(s,t)\in[0,1]^2\Big)$, respectively $\widetilde{W}^{\ve}:=\Big(\widetilde{W}^{(\ve,1)}_{s,t},\cdots,\widetilde{W}^{(\ve,d)}_{s,t};(s,t)\in[0,1]^2\Big)$ is the $\R^d$-valued two-parameter Gaussian process given by $W^{(\ve,i)}_{s,t}:=W^{(i)}_{s,a+\ve t}$, respectively $\widetilde{W}^{(i,\ve)}_{s,t}=W^{(i)}_{s,a+\ve t}-W^{(i)}_{s,a}$ for all $i\in\{1,\cdots,d\}$.
	\end{prop}	
	\begin{proof}
		The proof is based on the multidimensional local time-space calculus formula \eqref{eq:EisenSheetdD02} and the Barlow-Yor Inequality. We only prove \eqref{eq:DavieSheet1dd} when $s<\ps$. The proof in the case where $s=\ps$ and $y\neq(0,\ldots,0)$ follows the same lines.\\
		Fix $(a,\ve)\in[0,1]\times(0,1)$ and $i\in\{1,\cdots,d\}$.   Observe that for any $0<s<\ps\leq1$, $(\widetilde{W}^{\ve}_{s,t};0\leq t\leq1)$ and $(W^{\ve}_{\ps,t}-W^{\ve}_{s,t}; 0\leq t\leq1)$ are independent. Then, conditionally to $(W^{\ve}_{\ps,t}-W^{\ve}_{s,t}; 0\leq t\leq1)$, $g\Big(s,t,\widetilde{W}^{\ve}_{s,t},y+W^{\ve}_{\ps,t}-W^{\ve}_{s,t},0\leq t\leq1\Big)$ is a deterministic function of $(\widetilde{W}^{\ve}_{s,t},0\leq t\leq1)$. Hence, conditionally to $(W^{\ve}_{\ps,t}-W^{\ve}_{s,t}; 0\leq t\leq1)$, we may apply \eqref{eq:EisenSheetdD02} to the $d$-dimensional Brownian motion $(Y_{s,t}:=\ve^{-1/2}\widetilde{W}^{\ve}_{s,t};0\leq t\leq1)$ and to the function 
		$h_{\ps}:\,[0,1]^2\times\R^d\to\R$ given by 
		\begin{align*}
			h_{\ps}(s,t,x)=g\Big(s,t,\sqrt{\ve}\,x,y+W^{\ve}_{\ps,t}-W^{\ve}_{s,t}\Big).
		\end{align*}
		We obtain
		\begin{align*}
			&\frac{\sqrt{\ve s}}{|y|+\sqrt{\ps-s}}\int_0^1\partial_{x_i} g\Big(s,t,\widetilde{W}^{\ve}_{s,t},y+W^{\ve}_{\ps,t}-W^{\ve}_{s,t}\Big)\mathrm{d}t=frac{\sqrt{s}}{|y|+\sqrt{\ps-s}} \int_0^1\partial_{x_i} h_{\ps}(s,t,Y_{s,t})\mathrm{d}t\\
			=&-\frac{1}{|y|+\sqrt{\ps-s}}\int_0^1 h_{\ps}(s,t,Y_{s,t})\frac{d_tY^{(i)}_{s,t}}{\sqrt{s}}-\frac{1}{|y|+\sqrt{\ps-s}}\int_0^1h_{\ps}(s,1-t,Y_{s,1-t})\frac{\mathrm{d}_tB^{(i)}_{s,t}}{\sqrt{s}}\\
			&+\frac{1}{|y|+\sqrt{\ps-s}}\int_0^1\frac{h_{\ps}(s,1-t,Y_{s,1-t})Y^{(i)}_{s,1-t}}{\sqrt{s}(1-t)}\mathrm{d}t\\
			=&-\frac{1}{|y|+\sqrt{\ps-s}}\int_0^1g(s,t,\widetilde{W}^{\ve}_{s,t},y+W^{\ve}_{\ps,t}-W^{\ve}_{s,t})\frac{\mathrm{d}_tY^{(i)}_{s,t}}{\sqrt{s}}\\
			&-\frac{1}{|y|+\sqrt{\ps-s}}\int_0^1g(s,1-t,\widetilde{W}^{\ve}_{s,1-t},y+W^{\ve}_{\ps,1-t}-W^{\ve}_{s,1-t})\frac{\mathrm{d}_tB^{(i)}_{s,t}}{\sqrt{s}}\\
			&+\frac{1}{|y|+\sqrt{\ps-s}}\int_0^1\frac{g(s,1-t,\widetilde{W}^{\ve}_{s,1-t},y+W^{\ve}_{\ps,1-t}-W^{\ve}_{s,1-t})Y^{(i)}_{s,1-t}}{\sqrt{s}(1-t)}\mathrm{d}t=J_1+J_2+J_3.
		\end{align*}
		Using Jensen inequality, one has
		\begin{align*}
			&\E\Big[\exp\Big(\frac{\alpha \sqrt{\ve s}}{|y|+\sqrt{\ps-s}} \Big|\int_0^1\partial_{x_i} g\Big(s,t,\widetilde{W}^{\ve}_{s,t},y+W^{\ve}_{\ps,t}-W^{\ve}_{s,t}\Big)\mathrm{d}t\Big|\Big)\Big]\\
			\le&\frac{1}{3}\Big(\E\Big[\exp(3\alpha |J_1|)\Big]+\E\Big[\exp(3\alpha |J_2|)\Big]+\E\Big[\exp(3\alpha |J_3|)\Big]\Big).
		\end{align*}
		Hence, to get the desired estimate, it suffices to prove that, for every $k\in\{1,2,3\}$, there exist positive constants $\alpha_k$ and $C_k$ such that $\E\left[\exp(\alpha_k |J_k|)\right]\leq C_k.$
		Let us start with the estimate of $J_1$. We consider the martingale 
		\begin{align*}
			\Big(N^{(i)}_{t}:=\int_0^t\frac{g(s,u,\widetilde{W}^{\ve}_{s,u},y+W^{\ve}_{\ps,u}-W^{\ve}_{s,u})}{|y|+\sqrt{\ps-s}}\,\frac{\mathrm{d}_uY^{(i)}_{s,u}}{\sqrt{s}},t\in[0,1]\Big).
		\end{align*}
		For any constant $\alpha>0$, the following exponential expansion formula holds
		\begin{align*}
			\E[\exp(\alpha|J_1|)]=\E\Big[\exp\Big(\alpha\Big|N_{1}^{(i)}\Big|\Big)\Big]
			=1+\sum\limits_{m=1}^{\infty}\frac{\alpha^m\E\Big[\Big|N_{1}^{(i)}\Big|^{m}\Big]}{m!}.
		\end{align*}
		Applying the Barlow-Yor inequality to the martingale $\left(N_t^{(i)},t\in[0,1]\right)$ (see \cite{BY82}) and using \eqref{eq:BoundSheetdd}, 
		there exists a universal constant $c_1$ (not depending on $m$) such that,
		\begin{align*}
			\E\Big[\Big|N_{1}^{(i)}\Big|^{m}\Big]&\leq \E\Big[\sup\limits_{0\leq t\leq 1}|N_{t}^{(i)}|^{m}\Big]\leq c_1^{m}m^{m/2}\E\Big[\langle N^{(i)}\rangle_1^{m/2}\Big]\\
			&\leq c_1^{m}m^{m/2}\E\Big[\Big(\int_{0}^{1}\Big|\frac{g(s,t,\widetilde{W}^{\ve}_{s,t},y+W^{\ve}_{\ps,t}-W^{\ve}_{s,t})}{|y|+\sqrt{\ps-s}}\Big|^2\mathrm{d}t\Big)^{m/2}\Big]\\
			&\leq(c_1\kappa)^mm^{m/2}\E\Big[\Big(\int_0^1\Big|\frac{y+W^{\ve}_{\ps,t}-W^{\ve}_{s,t}}{|y|+\sqrt{\ps-s}}\Big|^2\mathrm{d}t\Big)^{m/2}\Big]. 
		\end{align*} 
		It follows from H\"older and triangle inequalities and the definition of $\widetilde{W}$ that
		\begin{align*}
			&\E\Big[\Big(\int_0^1\Big|\frac{y+W^{\ve}_{\ps,t}-W^{\ve}_{s,t}}{|y|+\sqrt{\ps-s}}\Big|^2\mathrm{d}t\Big)^{m/2}\Big]\leq\E\Big[\sup\limits_{0\leq t\leq 1}\Big|\frac{y+W^{\ve}_{\ps,t}-W^{\ve}_{s,t}}{|y|+\sqrt{\ps-s}}\Big|^m\Big]
			\\
			&\leq2^m\E\Big[1+\sup\limits_{0\leq t\leq 1}\Big|\frac{W^{\ve}_{\ps,t}-W^{\ve}_{s,t}}{\sqrt{\ps-s}}\Big|^m\Big]
			=2^m\E\Big[1+\sup\limits_{0\leq t\leq 1}\Big|\frac{\widetilde{W}^{\ve}_{\ps,t}-\widetilde{W}^{\ve}_{s,t}+W_{\ps,a}-W_{s,a}}{\sqrt{\ps-s}}\Big|^m\Big]\\
			\leq& (6d)^m\Big(1+\sum\limits_{i=1}^d\Big\{\E\Big[\sup\limits_{0\leq t\leq 1}\Big|\frac{\widetilde{W}^{(\ve,i)}_{\ps,t}-\widetilde{W}^{(\ve,i)}_{s,t}}{\sqrt{\ps-s}}\Big|^m\Big]+\E\Big[\Big|\frac{W^{(i)}_{\ps,a}-W^{(i)}_{s,a}}{\sqrt{\ps-s}}\Big|^m\Big]\Big\}\Big)
			\\
			\leq&(6d)^m\Big(1+\sum\limits_{i=1}^d\Big\{\ve^{m/2}\E\Big[\sup\limits_{0\leq t\leq 1}\Big|\frac{\widetilde{W}^{(\ve,i)}_{\ps,t}-\widetilde{W}^{(\ve,i)}_{s,t}}{\sqrt{\ve(\ps-s)}}\Big|^m\Big]+\E\Big[\sup\limits_{0\leq t\leq a}\Big|\frac{W^{(i)}_{\ps,t}-W^{(i)}_{s,t}}{\sqrt{\ps-s}}\Big|^m\Big]\Big\}\Big).
		\end{align*}
		Then, since for every $i\in\{1,\cdots,d\}$, $$\Big(\dfrac{\widetilde{W}^{(\ve,i)}_{\ps,t}-\widetilde{W}^{(\ve,i)}_{s,t}}{\sqrt{\ve(\ps-s)}},0\leq t\leq1\Big)\,\text{ and }\,\Big(\dfrac{W^{(i)}_{\ps,t}-W^{(i)}_{s,t}}{\sqrt{\ps-s}},0\leq t\leq1\Big)$$ are two standard Brownian motions starting from $0$, we deduce from Barlow-Yor inequality that 
		\begin{align*}
			\E\Big[\sup\limits_{0\leq t\leq 1}\Big|\frac{\widetilde{W}^{(\ve,i)}_{\ps,t}-\widetilde{W}^{(\ve,i)}_{s,t}}{\sqrt{\ve(\ps-s)}}\Big|^m\Big]\leq c_1^mm^{m/2}
		\end{align*}
		and
		\begin{align*}
			\E\Big[\sup\limits_{0\leq t\leq a}\Big|\frac{W^{(i)}_{\ps,t}-W^{(i)}_{s,t}}{\sqrt{\ps-s}}\Big|^m\Big]\leq\E\Big[\sup\limits_{0\leq t\leq 1}\Big|\frac{W^{(i)}_{\ps,t}-W^{(i)}_{s,t}}{\sqrt{\ps-s}}\Big|^m\Big]\leq c_1^mm^{m/2}.
		\end{align*}
		Thus,
		\begin{align*}
			\E\left[\exp\left(\alpha\left|J_1\right|\right)\right]
			=1+\sum\limits_{m=1}^{\infty}\frac{\alpha^m\E\left[\left|N_{1}^{(i)}\right|^{m}\right]}{m!}\leq1+2d\sum\limits_{m=1}^{\infty}\frac{(6d\alpha \kappa)^m(1+c_1)^{2m}m^{m}}{m!},
		\end{align*}
		which is finite for $\alpha<\dfrac{1}{6d\kappa e(1+c_1)^2}$. Then there exist positive constants $\alpha_1$ and $C_1$ that do not depend on $h$, $a$, $s$ and $\ps$ such that 
		\begin{align*}
			\E\left[\exp\left(\alpha_1\left|J_1\right|\right)\right]\leq C_1.
		\end{align*}
		The estimation of $J_2$ follows in an analogous manner. More precisely, if we consider the martingale
		\begin{align*}
			\Big(\widehat{N}^{(i)}_t:=\int_0^t\frac{g(s,1-u,\widetilde{W}^{\ve}_{s,1-u},y+W^{\ve}_{\ps,1-u}-W^{\ve}_{s,1-u})}{|y|+\sqrt{\ps-s}}\,\frac{\mathrm{d}_uB^{(i)}_{s,u}}{\sqrt{s}},t\in[0,1]\Big)
		\end{align*}
		then we deduce from the Barlow-Yor inequality that there exist positive constants $\alpha_2$ and $C_2$ such that
		\begin{align*}
			\E\left[\exp\left(\alpha_2\left|J_2\right|\right)\right]=\E\left[\exp\left(\alpha_2\left|\widehat{N}_1^{(i)}\right|\right)\right]\leq C_2.
		\end{align*}
		Next, we estimate $J_3$. Applying (\ref{eq:BoundSheetdd}) and Jensen inequality, we get 
		\begin{align*}
			\E\Big[\exp\Big(\frac{J_3}{16\kappa}\Big)\Big]
			\le&\E\Big[\exp\Big(\frac{1}{16\kappa}\int_{0}^{1}\frac{|g(s,1-t,\widetilde{W}^{\ve}_{s,1-t},y+W^{\ve}_{\ps,1-t}-W^{\ve}_{s,1-t})Y^{(i)}_{s,1-t}|}{(1-t)\sqrt{s}(|y|+\sqrt{\ps-s})}\mathrm{d}t\Big)\Big]\\
			\le&\E\Big[\exp\Big(\frac{1}{16}\int_{0}^{1}\frac{|y+W^{\ve}_{\ps,1-t}-W^{\ve}_{s,1-t}||Y^{(i)}_{s,1-t}|}{(1-t)\sqrt{s}(|y|+\sqrt{\ps-s})}\mathrm{d}t\Big)\Big]\\
			=&\E\Big[\exp\Big(\frac{1}{8}\int_{0}^{1}\Big|\frac{y+W^{\ve}_{\ps,1-t}-W^{\ve}_{s,1-t}}{|y|+\sqrt{\ps-s}}\Big|\Big|\frac{Y^{(i)}_{s,1-t}}{\sqrt{s(1-t)}}\Big|\frac{\mathrm{d}t}{2\sqrt{1-t}}\Big)\Big]\\
			\leq&\int_{0}^{1}\E\Big[\exp\Big(\frac{1}{8}\Big|\frac{y+W^{\ve}_{\ps,1-t}-W^{\ve}_{s,1-t}}{|y|+\sqrt{\ps-s}}\Big|\Big|\frac{Y^{(i)}_{s,1-t}}{\sqrt{s(1-t)}}\Big|\Big)\Big]\frac{\mathrm{d}t}{2\sqrt{1-t}}\\
			\leq&\int_{0}^{1}\E\Big[\exp\Big(\frac{1}{8}\Big(1+\Big|\frac{W^{\ve}_{\ps,1-t}-W^{\ve}_{s,1-t}}{\sqrt{\ps-s}}\Big|\Big)\Big|\frac{Y^{(i)}_{s,1-t}}{\sqrt{s(1-t)}}\Big|\Big)\Big]\frac{\mathrm{d}t}{2\sqrt{1-t}}.
		\end{align*}
		Moreover, by Cauchy-Schwarz inequality,	we have
		\begin{align*}		
			&\E\Big[\exp\Big(\frac{1}{8}\Big(1+\Big|\frac{W^{\ve}_{\ps,1-t}-W^{\ve}_{s,1-t}}{\sqrt{\ps-s}}\Big|\Big)\Big|\frac{Y^{(i)}_{s,1-t}}{\sqrt{s(1-t)}}\Big|\Big)\Big]	
			\\
			\leq&\E\Big[\exp\Big(\frac{1}{16}\Big(1+\Big|\frac{W^{\ve}_{\ps,1-t}-W^{\ve}_{s,1-t}}{\sqrt{\ps-s}}\Big|\Big)^2+\frac{1}{16}\Big|\frac{Y^{(i)}_{s,1-t}}{\sqrt{s(1-t)}}\Big|^2\Big)\Big]	\\
			\leq&\E\Big[\exp\Big(\frac{1}{8}+\frac{1}{8}\Big|\frac{W^{\ve}_{\ps,1-t}-W^{\ve}_{s,1-t}}{\sqrt{\ps-s}}\Big|^2+\frac{1}{16}\Big|\frac{Y^{(i)}_{s,1-t}}{\sqrt{s(1-t)}}\Big|^2\Big)\Big]\\
			=&\E\Big[\exp\Big(\frac{1}{8}+\frac{a+\ve(1-t)}{8}\Big|\frac{W^{\ve}_{\ps,1-t}-W^{\ve}_{s,1-t}}{\sqrt{(a+\ve(1-t))\ps-s}}\Big|^2+\frac{1}{16}\Big|\frac{Y^{(i)}_{s,1-t}}{\sqrt{s(1-t)}}\Big|^2\Big)\Big]\\
			\leq& \E\Big[\exp\Big(\frac{1}{8}+\frac{1}{4}\Big|\frac{W^{\ve}_{\ps,1-t}-W^{\ve}_{s,1-t}}{\sqrt{(a+\ve(1-t))\ps-s}}\Big|^2+\frac{1}{4}\Big|\frac{Y^{(i)}_{s,1-t}}{\sqrt{s(1-t)}}\Big|^2\Big)\Big].
		\end{align*}
		Observe that for every $t\in[0,1]$ fixed, the random variables 
		$$
		\dfrac{W^{(\ve,k)}_{\ps,1-t}-W^{(\ve,k)}_{s,1-t}}{\sqrt{(a+\ve(1-t))(\ps-s)}},\,k=1,\ldots,d\,\text{ and }\,\dfrac{Y_{s,1-t}^{(i)}}{\sqrt{s(1-t)}}
		$$ 
		are independent and normally distributed  with mean $0$ and variance $1$. Thus
		\begin{align*}
			&\E\Big[\exp\Big(\frac{1}{8}+\frac{1}{4}\Big|\frac{W^{\ve}_{\ps,1-t}-W^{\ve}_{s,1-t}}{\sqrt{(a+\ve(1-t))(\ps-s)}}\Big|^2+\frac{1}{4}\Big|\frac{Y_{s,1-t}^{(i)}}{\sqrt{s(1-t)}}\Big|^2\Big)\Big]\\
			&=e^{1/8}\prod\limits_{k=1}^d\E\Big[\exp\Big(\frac{1}{4}\Big|\frac{W^{(\ve,k)}_{\ps,1-t}-W^{(\ve,k)}_{s,1-t}}{\sqrt{(a+\ve(1-t))(\ps-s)}}\Big|^2\Big)\Big]\E\Big[\exp\Big(\frac{1}{4}\Big|\frac{Y_{s,1-t}^{(i)}}{\sqrt{s(1-t)}}\Big|^2\Big)\Big]<\infty. 
		\end{align*} 
		Hence, $\E\left[\exp\left(J_3/16\kappa\right)\right]$ is finite and, as a consequence, there exist positive constants $\alpha_3$ and $C_3$ (which do not depend on $a$, $\ve$, $s$ and $\ps$) such that $\E[\exp(\alpha_3|J_3|)]\leq C_3$. 
		This ends the proof.  
	\end{proof}
	
	\begin{corollary}\label{corol:DavieSheet1dds1}
		Let $b:\,[0,1]^2\times\R^d\to\R$ be a bounded Borel measurable function such that $\Vert b\Vert_{\infty}\leq1$. Let $W^{\ve}$ be defined as in Proposition \ref{prop:DavieSheet1dd}. Then for every $(a,\ve)\in[0,1]\times(0,1)$, every $(s,\ps,x,\xp)\in]0,1]^2\times\R^{2d}$ with $s\leq\ps$ and $(s,x)\neq(\ps,\xp)$, we have
		\begin{align}\label{eq:DavieSheet2dd}
			\E\left[\exp\left(\frac{\alpha \sqrt{\ve s}}{|\xp-x|+\sqrt{\ps-s}}\left|\int_0^1\left\{b(s,t,\xp+W^{\ve}_{\ps,t})-b(s,t,x+W^{\ve}_{s,t} )\right\}dt\right|\right)\right]\leq C,
		\end{align}
		where $\alpha$ and $C$ are the constants in Proposition \ref{prop:DavieSheet1dd}.
	\end{corollary}	
	
	\begin{proof}
		
		We first suppose that $b$ is differentiable and compactly supported.
		By the fundamental theorem of vector calculus, we obtain
		\begin{align*}
			&\Big|\int_0^1\{b(s,t,\xp+W^{\ve}_{\ps,t})-b(s,t,x+W^{\ve}_{s,t})\}\mathrm{d}t\Big|\\
			=&\Big|\int_0^1\int_0^1\nabla_{x}b_{}(s,t,x+W^{\ve}_{s,t}{} +u(\xp-x+W^{\ve}_{\ps,t}-W^{\ve}_{s,t}))\cdot(\xp-x+W^{\ve}_{\ps,t}-W^{\ve}_{s,t})\mathrm{d}u \mathrm{d}t\Big|\\
			\leq&\sum\limits_{i=1}^d\Big|\int_0^1\int_0^1\frac{\partial b}{\partial x_i}(s,t,\widetilde{W}^{\ve}_{s,t}+x+W_{s,a}{} +u(\xp-x+W^{\ve}_{\ps,t}-W^{\ve}_{s,t}))(\xp_i-x_i+W^{\ve,i}_{\ps,t}-W^{\ve,i}_{s,t})\mathrm{d}u \mathrm{d}t\Big|\\
			\le&\sum\limits_{i=1}^d\Big|\int_0^1\int_0^1\frac{\partial \widehat{b}_u}{\partial x_i}(s,t,\widetilde{W}^{\ve}_{s,t}+W_{s,a},\xp-x+W^{\ve}_{\ps,t}-W^{\ve}_{s,t})\mathrm{d}u \mathrm{d}t\Big|, 
		\end{align*}
		where for every $u\in[0,1]$, $\widehat{b}_u:\,[0,1]^2\times\R^{2d}$ is the function defined by 
		$$\widehat{b}_u(s,t,y,z)=z_ib(s,t,x+ y+uz).$$
		Since $(\widetilde{W}^{\ve}_{s,t},t\in[0,1])$ and $W_{s,a}$ are independent, it follows from Jensen inequality and \eqref{eq:DavieSheet1dd} applied to the function $g_u:\,(s,t,y,z)\longmapsto \widehat{b}_u(s,t,y+\zeta,z)$ that there exist positive constants $\alpha$ and $C$ such that
		\begin{align}
			&\E\Big[\exp\Big(\frac{\alpha \sqrt{\ve s}}{|\xp-x|+\sqrt{\ps-s}}\Big|\int_0^1\left\{b(s,t,\xp+W^{\ve}_{\ps,t}{} )-b(s,t,x+W^{\ve}_{s,t}{} )\right\}\mathrm{d}t\Big|\Big)\Big]\nonumber\\
			\le&\frac{1}{d}\sum\limits_{i=1}^d\int_0^1\E\Big[\exp\Big(\frac{\alpha \sqrt{\ve s}}{|\xp-x|+\sqrt{\ps-s}}\int_0^1\frac{\partial \widehat{b}_u}{\partial x_i}(s,t,\widetilde{W}^{\ve}_{s,t}+W_{s,a},W^{\ve}_{\ps,t}-W^{\ve}_{s,t}+\xp-x)\mathrm{d}t\Big)\Big]\mathrm{d}u\nonumber\\
			\le&\frac{1}{d}\sum\limits_{i=1}^d\int_0^1\int_{\R}\E\Big[\exp\Big(\frac{\alpha d\sqrt{\ve s}}{|\xp-x|+\sqrt{\ps-s}}\Big|\int_0^1\frac{\partial g_u}{\partial x_i}(s,t,\widetilde{W}^{\ve}_{s,t}+\zeta,W^{\ve}_{\ps,t}-W^{\ve}_{s,t}+\xp-x)\mathrm{d}t\Big|\Big)\Big]\Pb_{W_{s,a}}(\mathrm{d}\zeta) \mathrm{d}u\nonumber\\
			\leq &C.\label{eq:DavieSheetEst1dd}
		\end{align}
		When $b$ is not differentiable, then, since the set of compacly supported and differentiable functions is dense in $L^{\infty}([0,1]^2\times\R^d)$, there exists a sequence $(b_n,n\in\N)$ of compactly supported and differentiable functions which converges a.e. to $b$ on $[0,1]^2\times\R^d$, and the desired result follows from the Vitali's convergence theorem. The proof is completed.
	\end{proof}
	
	\begin{corollary}\label{corol:DavieSheetdd34}
		Let $b:\,[0,1]\times\R^d\to\R$ be a Borel measurable function such that $|b(t,x)|\leq1$ everywhere on $[0,1]\times\R^d$.  
		For $(s,\ps,x,\xp)\in]0,1]^2\times\R^{2d}$ with $s\leq\ps$, $(s,x)\neq(\ps,\xp)$, $0\leq a< \pa\leq 1$ and $(x,\xp)\in\R^{2d}$, define 
		\begin{align*}
			\rho(s,x;\ps,\xp):=\int_a^{\pa}\left\{b(t,\xp+W_{\ps,t})-b(t,x+W_{s,t})\right\}\mathrm{d}t. 
		\end{align*}
		Then, for every $(s,\ps)\in]0,1]^2$ and every $\eta>0$,
		\begin{align}\label{eq:EstDavieSigma1dd1}
			\Pb\left(\sqrt{s}|\rho(s,x;\ps,\xp)|\geq\eta\sqrt{\ve}(|\xp-x|+\sqrt{\ps-s})\right)
			\leq Ce^{-\alpha\eta},
		\end{align}
		where $\ve=\pa-a$.
	\end{corollary}	
	
	\begin{proof}
		Use the change of variable $v=a+\ve t$ to obtain
		\begin{align*}
			\rho(s,x;\ps,\xp)&=\int_a^{\pa}\left\{b(v,\xp+W_{\ps,v})-b(v,x+W_{s,v})\right\}\mathrm{d}v\\
			&=\ve\int_0^1\left\{b(a+\ve t,\xp+W_{\ps,a+\ve t}{} )-b(a+\ve t,x+W_{s,a+\ve t}{} )\right\}\mathrm{d}t\\
			&=\ve\int_0^1\left\{\widetilde{b}_{\ve}(t,\xp+W^{\ve}_{\ps,t}{} )-\widetilde{b}_{\ve}(t,x+W^{\ve}_{s,t}{} )\right\}\mathrm{d}t,
		\end{align*}
		where $\ve=\pa-a$ and $\widetilde{b}_{\ve}:\,(t,x)\longmapsto b(a+\ve t,x)$.  
		Hence, by (\ref{eq:DavieSheet2dd}) and Chebychev inequality, we have
		\begin{align*}
			&\Pb\Big(\sqrt{s}|\rho(s,x;\ps,\xp)|>\eta\sqrt{\ve}(|\xp-x|+\sqrt{\ps-s})\Big)\\
			&=\Pb\Big(\frac{\alpha\sqrt{\ve s}}{|\xp-x|+\sqrt{\ps-s}}\Big|\int_0^1\left\{\widetilde{b}_{\ve}(t,\xp+W^{\ve}_{\ps,t}{} )-\widetilde{b}_{\ve}(t,x+W^{\ve}_{s,t}{} )\right\}\mathrm{d}t \Big|>\alpha\eta\Big)\leq Ce^{-\alpha\eta}, 
		\end{align*}		
		The proof is completed.  
	\end{proof}

	\subsection{On the regularising properties of Brownian sheet paths}
	For any Bounded Borel measurable function $b:\,\R_+\times\R^d\to\R$ such that $|b(t,x)|\leq 1$ everywhere, any nonnegative integers $n$, $k$, we define the averaging type operator $T^W_{I_{nk}}[b]$ defined by
	
	\begin{align*}
		T^W_{I_{nk}}[b](s,x)=\int_{I_{nk}}b(t,W_{s,t}+x)\mathrm{d}t,\quad\forall\,s\in\R_+,
	\end{align*}
	where $I_{nk}=[k2^{-n},(k+1)2^{-n}]$. Define also
	\begin{align*}
		\rho_{nk}(s,x;\ps,\xp)=T^W_{I_{nk}}[b](\ps,\xp)-T^W_{I_{nk}}[b](s,x)=\int_{I_{nk}}\{b(t,W_{\ps,t}+\xp)-b(t,W_{s,t}+x)\}\mathrm{d}t,\quad\forall\,(s,\ps)\in\R_+^2.
	\end{align*}	
	
	The aim of this section is to provide a continuity property of the operator  $T^W_{I_{nk}}[b]$. In particular, we derive the following result which gives the modulus of continuity of $\rho_{nk}$.
	
	\begin{theorem}\label{theo:Pseudometric}
		Let $b$ satisfy conditions of Corollary \ref{corol:DavieSheetdd34}. Then there exists a subset $\Omega_{0}$ of $\Omega$ with $\Pb(\Omega_{0})=1$ and a positive random constant $C_0$ such that 
		\begin{align}\label{Pseudometric}
			|\rho_{nk}(s,x;\ps,\xp)(\omega)|\leq \frac{C_0(\omega)2^{-n/2}}{\sqrt{s}}\Big[n+ \log^+\frac{1}{|\xp-x|+\sqrt{\ps-s}}\Big]\Big(|\xp-x|+\sqrt{\ps-s}\Big)
		\end{align}
		for all $\omega\in\Omega_{0}$,  all $(s,\ps,x,\xp)\in]0,1]^2\times[-1,1]^{2d}$ with $s\leq\ps,\,(\ps,\xp)\neq (s,x)$ and all choices of integers $n,\,k$ with $n\geq1$, $0\leq k\leq 2^n-1$, where $C_0$  does not depend on $n$, $k$, $s$ and $\ps$.
	\end{theorem}
	
	\begin{remark}
		The factor $1/\sqrt{s}$ on the right side of \eqref{Pseudometric} shows that the averaging operator $T^{W}_{I}[b]$ has a singularity of order $-1/2$ at $s=0$. This is a consequence of the scaling and boundary properties of the Brownian sheet.
	\end{remark}
	
	\begin{corollary}\label{corol:Pseudometric}
		Let $b$  satisfy conditions of Corollary \ref{corol:DavieSheetdd34}. There exist $\Omega_1\subset\Omega$ with $\Pb(\Omega_1)=1$ and a positive random constant $C_1=C_1(\omega)$ such that 
		\begin{align*}
			\left|T^W_I[b](\ps,\xp)-T^W_I[b](s,x)\right|\leq \frac{C_1\sqrt{|I|}}{\sqrt{s}}    \Big[1+\log^+\frac{1}{(|\xp-x|+\sqrt{\ps-s})|I|}\Big]\Big(|\xp-x|+\sqrt{\ps-s}\Big),
		\end{align*}
		for all $\omega\in\Omega_1$, all $(s,\ps,x,\xp)\in]0,1]^2\times[-1,1]^{2d}$ with $s\leq\ps,\,(\ps,\xp)\neq (s,x)$, all sub-interval $I$ of $[0,1]$, where $|I|$ denotes the length of $I$ and $C_1$  does not depend on $n$, $k$, $\ps$ and $\xp$.
	\end{corollary}		
	The next result follows immediately from \eqref{Pseudometric} when $s=\ps$ by integrating with respect to $s$ over $I_{n\ell}$.	
	\begin{corollary}\label{corol:Pseudometric2}
		Let $b$  satisfy conditions of Corollary \ref{corol:DavieSheetdd34}. There exist $\Omega_2\subset\Omega$ with $\Pb(\Omega_2)=1$ and a positive random constant $C_2=C_2(\omega)$ such that 
		\begin{align*}
			\left|\int_{I_{n\ell}}T^W_{I_{nk}}[b](s,\xp)\mathrm{d}s-\int_{I_{n\ell}}T^W_{I_{nk}}[b](s,x)\mathrm{d}s\right|\leq  C_22^{-n}    \Big(n+\log^+\frac{1}{|\xp-x|}\Big)|\xp-x|,
		\end{align*}
		for all $\omega\in\Omega_2$, all $(x,\xp)\in[-1,1]^{2d}$ and all choices of integers $n$, $k$, $\ell$, $n\geq1$, $0\leq k,\ell\leq 2^n-1$.
	\end{corollary}
	
	The proof of Theorem \ref{theo:Pseudometric} is done in two steps. We start by proving the when $(s,x)$ and $(\ps,\xp)$ are dyadic couples. Then the desired estimate will follow from a density property of the set of dyadic numbers and a continuity lemma.
	
	\begin{lemma}\label{lem:PseudoMetric1}
		For every $\ve>0$, there exist a positive deterministic constant $C_{\ve}$ such that,  for any 
		$b$ satisfying conditions of Corollary \ref{corol:DavieSheetdd34}, we can find $\Omega_{\ve}\subset\Omega$ with $\Pb(\Omega_{\ve})\geq1-\ve/2$ and   
		\begin{align*}
			|\rho_{nk}(s,x;\ps,\xp)(\omega)|\leq \frac{C_{\ve}2^{-n/2}}{\sqrt{s}}\Big[n+ \log^+\frac{1}{|\xp-x|+\sqrt{\ps-s}}\Big]\Big(|\xp-x|+\sqrt{\ps-s}\Big)
		\end{align*}
		for all $\omega\in\Omega_{\ve}$, all dyadic quadruples $(s,\ps,x,\xp)\in]0,1]^2\times[-1,1]^{2d}$ with $s\leq\ps,\,(\ps,\xp)\neq (s,x)$ and all choices of integers $n,\,k$ with $n\geq1$, $0\leq k\leq 2^n-1$. 
	\end{lemma}
	\begin{proof}
		Denote by $\mathcal{Q}$ be the set of dyadic quadruples  $(s,\ps,x,\xp)\in]0,1]^2\times[-1,1]^{2d}$ such that $s\leq\ps$ and $(s,x)\neq(\ps,\xp)$.  Let us define 
		$$
		\mathcal{Q}_{m}=\left\{(s,\ps,x,\xp)\in]0,1]^2\times[-1,1]^{2d}:\,s\leq\ps,\,(\ps,\xp)\neq (s,x),\,4^m(s,\ps)\in\N^2\text{ and }2^m(x,\xp)\in\Z^{2d}\right\}.
		$$ 
		Observe that $\mathcal{Q}_{m}$ does not have more than $2^{4d}2^{6dm}$ elements and it holds $\mathcal{Q}=\bigcup\limits_{m\in\N}\mathcal{Q}_{m}$.	
		For every $n\in\N$, $\delta\in\Q_+$, consider the sets
		\begin{align*}
			\mathbf{E}_{\delta,n}=&\left\{\omega\in\Omega:\,\text{there exist }k\in\{0,1,\cdots,2^n-1\},\text{ }m\in\N^{\ast},\text{ } \text{ and }(s,\ps,x,\xp)\in\mathcal{Q}_{m}{}_{}^{}\right.\\
			&\quad \text{ such that }
			\sqrt{s}|\rho_{nk}(s,x;\ps,\xp)|(\omega)\geq\delta(1+n+m)2^{-n/2}(|\xp-x|+\sqrt{\ps-s})\}.
		\end{align*}
		Then
		\begin{align*}
			\mathbf{E}_{\delta,n}=
			\bigcup_{k=0}^{2^n-1}\bigcup_{m=1}^{\infty}\Big(\bigcup_{(s,\ps,x,\xp)\in\mathcal{Q}_{m}}\left\{\omega\in\Omega:\,\sqrt{s}|\rho_{nk}(s,x;\ps,\xp)|(\omega)\geq\delta(1+n+m)2^{-n/2}(|\xp-x|+\sqrt{\ps-s})\right\}\Big).
		\end{align*}	 
		Let  $\mathbf{E}_{\delta}$ be defined as
		$$\mathbf{E}_{\delta}:=\bigcup\limits_{n=0}^{\infty}\mathbf{E}_{\delta,n}.$$
		We deduce from \eqref{eq:EstDavieSigma1dd1} that 
		\begin{align*}
			\Pb(\mathbf{E}_{\delta})&\leq\sum\limits_{n=0}^{\infty}\sum\limits_{k=0}^{2^n-1}\sum\limits_{m=0}^{\infty} \sum\limits_{(s,\ps,x,\xp)\in\mathcal{Q}_{m}}\Pb\Big(\sqrt{s}|\rho_{nk}(s,x;\ps,\xp)|\geq \delta(1+n+m)2^{-n/2}(|\xp-x|+\sqrt{\ps-s})\Big)\\
			&\leq C2^{4d} \sum\limits_{n=0}^{\infty}\sum\limits_{m=0}^{\infty}2^{n}2^{6dm}e^{-\alpha\delta(1+n+m)}.
		\end{align*}
		In particular for $\delta\geq\delta_0:=\alpha^{-1}(6d+1)$, we have
		\begin{align}\label{eq:EstlimPEdelta0e}
			\lim\limits_{\delta\to\infty}\Pb(\mathbf{E}_{\delta})\leq\lim\limits_{\delta\to\infty} C2^{4d}e^{-\alpha\delta}\,\sum\limits_{n=0}^{\infty}\sum\limits_{m=0}^{\infty}2^{-6dn-m}\leq \lim\limits_{\delta\to\infty}4C2^{4d}e^{-\alpha\delta}=0.
		\end{align}
		Hence for every $\ve>0$, there exists $\delta_{\ve}\in\Q_+$ such that $\Pb(\mathbf{E}_{\delta_{\ve}})<\ve/2$. Choose $\Omega_{\ve}=\Omega\setminus\mathbf{E}_{\delta_{\ve}}$.
		Then $\Pb(\Omega_{\ve})=1-\ve/2$ and for every $\omega\in\Omega_{\ve}$,  
		\begin{align}\label{eq:DaviePseudoMetric00}
			|\rho_{nk}(s,x;\ps,\xp)(\omega)|<\frac{\delta_{\ve} (1+n+m)2^{-n/2}}{\sqrt{s}}(|\xp-x|+\sqrt{\ps-s})
		\end{align}
		for all choices of $n$, $k$, $m$ and  $(s,\ps,x,\xp)\in\mathcal{Q}_{m}$.
		
		Now, choose any quadruple $(s,\ps,x,\xp)\in\mathcal{Q}$ and let $m$ be the smallest nonnegative integer $m$ such that $2^{-m-1}\leq |\xp-x|+\sqrt{\ps-s}$. For $r\geq m$ and for every $i\in\{1,\ldots,d\}$, define the sequences 
		\begin{align*}
			s_r=1-4^{-r}[4^r(1-s)],\,\ps_r=1-4^{-r}[4^r(1-\ps)],\,x_{i,r}=2^{-r}[2^rx_i]\text{ and }\xp_{i,r}=2^{-r}[2^r\xp_i],
		\end{align*}
		where $[\cdot]$ denotes the integer part function. Observe that for every $(\alpha,\beta)\in\R^2$, we have $|[\alpha]-[\beta]|\leq1+|\alpha-\beta|$, $0\leq[4\alpha]-4[\alpha]\leq3$ and $0\leq[2\alpha]-2[\alpha]\leq1$, then it holds
		\begin{align*}
			|s_m-\ps_m|\leq 2\times4^{-m},\text{ } |\xp_m-x_m|\leq\sqrt{d}\,2^{1-m},
		\end{align*}
		and for every $r\geq m$, we obtain
		\begin{align*}
			|s_{r+1}-s_r|\leq3\times4^{-r-1},\,|\ps_{r+1}-\ps_r|\leq3\times4^{-r-1},\,|x_{r+1}-x_r|\leq\sqrt{d}\,2^{-r-1}\text{ and }|\xp_{r+1}-\xp_r|\leq\sqrt{d}\,2^{-r-1}.
		\end{align*}	
		It follows from the definition of $\rho_{nk}$ that 
		\begin{align}\label{eq:AddPropRho}
			\rho_{nk}(s,x;\ps,\xp)=\rho_{nk}(s,x;s_m,x_m)+\rho_{nk}(s_m,x_m;\ps_m,\xp_m)+\rho_{nk}(\ps_m,\xp_m;\ps,\xp).
		\end{align}
		Moreover for every integer $q\geq m+1$, we have
		\begin{align*}
			\rho_{nk}(\ps_{q+1},\xp_{q+1};\ps_m,\xp_m)=\sum\limits_{r=m}^{q}\rho_{nk}(\ps_{r+1},\xp_{r+1};\ps_{r},\xp_r)\,\text{ and }\\\rho_{nk}(s_m,x_m;s_{q+1},x_{q+1})=\sum\limits_{r=m}^{q}\rho_{nk}(s_{r},x_r;s_{r+1},x_{r+1}),
		\end{align*}
		
		from which we deduce that
		\begin{align*}
			\rho_{nk}(\ps,\xp;\ps_m,\xp_m)=\sum\limits_{r=m}^{\infty}\rho_{nk}(\ps_r,\xp_r;\ps_{r+1},\xp_{r+1})\,\text{ and }\rho_{nk}(s_m,x_m;s,x)=\sum\limits_{r=m}^{\infty}\rho_{nk}(s_{r+1},x_{r+1};s_{r},x_r).
		\end{align*}
		The above equalities come from the fact that for some integer $q\geq m+1$, we have $\ps_r=\ps$, $s_r=s$, $x_r=x$ and $\xp_r=\xp$ for all $r\geq q$. Using  \eqref{eq:DaviePseudoMetric00}, \eqref{eq:AddPropRho} and the fact that $\ps_{r}\geq s_{r}\geq s$ for $r\geq m$, we obtain that for every $\omega\in\Omega_{\ve}$,
		\begin{align*}
			&2^{n/2}\,|\rho_{nk}(s,x;\ps,\xp)(\omega)|\\&\leq2^{n/2}\left(|\rho_{nk}(s_m,x_m;\ps_m,\xp_m)(\omega)|+\sum\limits_{r=m}^{\infty}|\rho_{nk}(s_{r+1},x_{r+1};s_{r},x_r)(\omega)|+\sum\limits_{r=m}^{\infty}|\rho_{nk}(\ps_r,\xp_r;\ps_{r+1},\xp_{r+1})(\omega)|\right)\\
			&\leq4\sqrt{d}\delta_{\ve}\frac{(1+n+m)}{\sqrt{s_m}}2^{-m}+4\sqrt{d}\delta_{\ve}\sum\limits_{r=m}^{\infty}\frac{(2+n+r)}{\sqrt{s_{r+1}}}2^{-r-1}+4\sqrt{d}\delta_{\ve}\sum\limits_{r=m}^{\infty}\frac{(2+n+r)}{\sqrt{\ps_{r+1}}}2^{-r-1}\\
			&\leq\frac{4\sqrt{d}\delta_{\ve}(1+n+m)}{\sqrt{s}}2^{-m}+\frac{8\sqrt{d}\delta_{\ve}(n+1)}{\sqrt{s}}\sum\limits_{r=m}^{\infty}2^{-r-1}+\frac{8\sqrt{d}\delta_{\ve}}{\sqrt{s}}\sum\limits_{r=m}^{\infty}(r+1)2^{-r-1}.
		\end{align*} 
		Using the facts that $\sum\limits_{r=m+1}^{\infty}2^{-r}= 2^{-m}$, $\sum\limits_{r=m+1}^{\infty}r2^{1-r}= (2+m)2^{1-m}$ and  $2^{-m-1}\leq|\xp-x|+\sqrt{\ps-s}<2^{-m}$,
		we obtain
		\begin{align*}
			2^{n/2}\,|\rho_{nk}(\ps,\xp;s,x)|\leq& \frac{36\sqrt{d}\delta_{\ve}(n+m+1)}{\sqrt{s}}2^{-m}\\ \leq &\frac{108d\delta_{\ve}}{\sqrt{s}} \Big[1+n+\log^+\frac{1}{|\xp-x|+\sqrt{\ps-s}}\Big]\Big(|\xp-x|+\sqrt{\ps-s}\Big). 
		\end{align*} 
		The result follows by taking $C_{\ve}=108d\delta_{\ve}$.
	\end{proof}		
	Subsequently, $\mathbf{1}_U$ denotes the indicator function of a given Borel set $U$ and $|U|$ denotes the Lebesgue measure of $U$, when there is no confusion.
	
	\begin{lemma}\label{lem:IntIndicFuncEstim}
		For all $\ve>0$, there exists $\eta_{\ve}>0$ such that, if $U\subset(0,1)\times\R^d$ is open and satisfies $|U|<\eta_{\ve}$, then
		\begin{align*}
			\Pb\Big(\Big\{\sqrt{s}\int_0^1\mathbf{1}_{U}(t,W_{s,t}+x)dt\leq\ve,\,\forall\,(s,x)\in]0,1]\times[-1,1]^d\Big\}\Big)\geq1-\ve.
		\end{align*}
	\end{lemma}	
	\begin{proof}
		For every $n\in\N^{\ast}$, let
		\begin{align*}
			\mathcal{O}_{n}=\{(s,\ps,x,\xp)\in\mathcal{Q}:\,|\xp-x|+\sqrt{\ps-s}\leq\sqrt{d}\,2^{1-n}\}.
		\end{align*}
		It follows from Lemma \ref{lem:PseudoMetric1} that for every $\ve>0$, there exists a deterministic constant $C_{\ve}$ such that for any Borel measurable real valued function $\varphi$ on $[0,1]\times\R^d$ satisfying $|\varphi|\leq1$, we can find $\Omega_{\ve}$ with $\Pb(\Omega_{\ve})\geq1-\ve/2$ and 
		\begin{align*}
			\sqrt{s}|\rho^{(\varphi)}_{nk}(s,x;\ps,\xp)(\omega)|=&\sqrt{s}\int_{I_{nk}}\left\{\varphi(t,\xp+W_{\ps,t})-\varphi(t,x+W_{s,t})\right\}\mathrm{d}t\\
			\le&C_{\ve}2^{-n/2}\Big(n+\log^{+}\dfrac{1}{|\xp-x|+\sqrt{\ps-s}}\Big)(|\xp-x|+\sqrt{\ps-s})\\
			\le& 2\sqrt{d}\,C_{\ve}\Big(1+\log^+\frac{1}{|\xp-x|+\sqrt{\ps-s}}\Big)(|\xp-x|+\sqrt{\ps-s})^{1/3}n2^{-7n/6}\leq K_{\ve}n2^{-7n/6}
		\end{align*}
		for every $\omega\in\Omega_{\ve}$, every $(n,k)\in\N^2$, $0\leq k\leq2^n-1$ and every $(s,\ps,x,\xp)\in\mathcal{O}_n$,  
		where $K_{\ve}=2\sqrt{d}\,C_{\ve}\sup\limits_{\zeta\in]0,1]}\left[1+\log^+(1/\zeta)\right]\zeta^{1/3}.$
		Now, take $m$ such that $K_{\ve}\sum\limits_{k=m}^{\infty}n2^{-n/6}<\ve/2$. Let us consider the set of dyadic numbers 
		$$\mathcal{J}_m=\left\{k4^{-m}:\,k=0,1,\ldots,4^{m}\right\}\times\left\{-1+\ell2^{-m}:\,\ell=0,1,2,\cdots,2^{m+1}\right\}^d.$$ 
		Let $\delta$ be small enough. Then for any bounded Borel measurable function $\varphi:\,[0,1]\times\R^d\to\R$ such that $\Vert \varphi\Vert_{L^{2d}([0,1]\times\R^d)}<\delta$, it holds that 
		\begin{align*}
			\Pb\Big(\Big|\int_{I_{mk}}\sqrt{s}\,\varphi(t,x+W_{s,t})\mathrm{d}t\Big|\geq\frac{\ve}{2^{m+2}}\Big)\leq\frac{\ve}{2^{m+1}\#(\mathcal{J}_m)}
		\end{align*}
		for some $k$ and $(s,x)\in\mathcal{J}_m$, where $\#(\mathcal{J}_m)$ denotes the number of elements in $\mathcal{J}_m$. Indeed, if we set $g(s,t,y)=(2\pi st)^{-d/2}e^{-|y|^2/2st}$, then by Markov inequality and H\"older inequality, we have
		\begin{align*}
			&\Pb\Big(\Big|\int_{I_{mk}}\sqrt{s}\,\varphi(t,x+W_{s,t})\mathrm{d}t\Big|
			\geq\frac{\ve}{2^{m+2}}\Big)\\ \leq&
			(2^{m+2}/\ve)\E\Big[\Big|\int_{I_{mk}}\sqrt{s}\,\varphi(t,x+W_{s,t})\mathrm{d}t\Big|\Big]\\
			\leq&(2^{m+2}\sqrt{s}/\ve)\int_0^1\int_{\R^d}|\varphi(t,x+y)|g(s,t,y)\,\mathrm{d}y\mathrm{d}t\\
			\leq&(2^{m+2}\sqrt{s}/\ve)\Big(\int_0^1\int_{\R^d}|\varphi(t,x+y)|^{2d}\,\mathrm{d}y\mathrm{d}t\Big)^{\frac{1}{2d}}\Big(\int_0^1\int_{\R^d}g(s,t,y)^{\frac{2d}{2d-1}}\,\mathrm{d}y\mathrm{d}t\Big)^{1-\frac{1}{2d}}\\ \leq&(2^{m+2}c(d)/\ve)\|\varphi\|_{L^{2d}([0,1]\times\R^d)}\leq(2^{m+2}c(d)/\ve)\delta,
		\end{align*}
		since $\int_{\R^d}g(s,t,y)^{\frac{2d}{2d-1}}\,\mathrm{d}y=c(d)(st)^{-\frac{d}{2(2d-1)}}$, where $c(d)$ is a constant that depends on $d$.\\
		Hence, it suffices to take $\delta$ such that $\delta<\ve^22^{-m-3}/c(d)\#(\mathcal{J}_m)$ to get the above claim. Therefore, we have
		\begin{align*}
			\Pb\Big(\Big|\int_{I_{mk}}\sqrt{s}\,\varphi(t,x+W_{s,t})\mathrm{d}t\Big|<\frac{\ve}{2^{m+2}},\,\forall\,k\in\{0,1,\cdots,2^m-1\},\,\forall\,(s,x)\in\mathcal{J}_m\Big)
			\leq1-\ve/2.
		\end{align*}
		Let $U\subset[0,1]\times\R^d$ be an open set such that $|U|<\eta:=\delta^{2d}$ and consider a non-decreasing sequence $(\varphi_r)_{r\in\N}$ of continuous nonnegative functions on $[0,1]\times\R^d$ that converges pointwise to $\mathbf{1}_{U}$. Observe that for every $r\in\N$, $\Vert \varphi_r\Vert_{L^{2d}([0,1]\times\R^d)}<\delta$ since $0\leq\varphi_r\leq\mathbf{1}_{U}$. Let $r\in\N$ and define the events $A_r$ and $B_r$ by 
		$$
		A_r:=\Big\{\Big|\int_{I_{mk}}\sqrt{s}\,\varphi_r(t,x+W_{s,t})\mathrm{d}t\Big|<\frac{\ve}{2^{m+2}},\,\forall\,k,\,\forall\,(s,x)\in\mathcal{J}_m\Big\}
		$$
		and
		\begin{align*}
			&B_r:=
			\Big\{\int_{I_{nk}}\sqrt{s}\,\left[\varphi_r(t,\xp+W_{\ps,t})-\varphi_r(t,x+W_{s,t})\right]\mathrm{d}t\leq K_{\ve}n2^{-7n/6},\,\forall\,n\in\N^{\ast},\,\forall\,k,\,\forall\,(s,\ps,x,\xp)\in \mathcal{O}_n\Big\}.
		\end{align*}
		Then $\Pb(A_r)\geq1-\ve/2$, $\Pb(B_r)\geq1-\ve/2$,  $\Pb(A_r\cap B_r)\geq1-\ve$ and for any $\omega\in A_r\cap B_r$,
		\begin{align}\label{eq:Contlemma1}
			\Big|\int_0^1\sqrt{s}\,\varphi_r(t,x+W_{s,t})(\omega)\mathrm{d}t\Big|\leq\sum\limits_{k=0}^{2^m-1}	\Big|\int_{I_{mk}}\sqrt{s}\,\varphi_r(t,x+W_{s,t})(\omega)\mathrm{d}t\Big|<\ve/2,\quad\forall\,(s,x)\in\mathcal{J}_m,
		\end{align}
		and 
		\begin{align}\label{eq:Contlemma2}
			&\Big|\int_0^1\sqrt{s}\,\{\varphi_r(t,\xp+W_{\ps,t})-\varphi_r(t,x+W_{s,t})\}(\omega)\mathrm{d}t\Big|\\&\leq\sum\limits_{k=0}^{2^n-1}	\Big|\int_{I_{nk}}\sqrt{s}\,\{\varphi_r(t,\xp+W_{\ps,t})-\varphi_r(t,x+W_{s,t})\}(\omega)\mathrm{d}t\Big|\leq K_{\ve}n2^{-n/6},\quad\forall\,(s,\ps,x,\xp)\in\mathcal{O}_n.\nonumber
		\end{align}  
		For a fixed $s\in]0,1]$, let $(s_n,n\in\N)$ and $(x_n,n\in\N)$ be given respectively by $s_n=1-4^{-n}[4^n(1-s)]$ and $x_{i,n}=2^{-n}[2^nx_i]$ for every $i\in\{1,\cdots,d\}$ and $n\in\N$. Since $(s_m,x_m)\in\mathcal{J}_m$,  $(s_{n+1},s_{n},x_n,x_{n+1})\in\mathcal{O}_n$ and $s_n\geq s$,
		we deduce from \eqref{eq:Contlemma1} and \eqref{eq:Contlemma2} that for any $\omega\in A_r\cap B_r$,
		\begin{align*}
			&\Big|\sqrt{s}\,\int_0^1\varphi_r(t,x+W_{s,t})(\omega)\mathrm{d}t\Big|\\
			\le&\sqrt{s}\,\Big|\int_0^1\varphi_r(t,x_m+W_{s_m,t})(\omega)\mathrm{d}t\Big|+\sum\limits_{n=m}^{\infty}\sqrt{s}\,\Big|\int_0^1\{\varphi_r(t,x_{n+1}+W_{s_{n+1},t})-\varphi_r(t,x_n+W_{s_n,t})\}(\omega)\mathrm{d}t\Big|\\
			&\leq\sqrt{s_m}\,\Big|\int_0^1\varphi_r(t,x_m+W_{s_m,t})(\omega)\mathrm{d}t\Big|+\sum\limits_{n=m}^{\infty}\sqrt{s_{n+1}}\,\Big|\int_0^1\{\varphi_r(t,x_{n+1}+W_{s_{n+1},t})-\varphi_r(t,x_n+W_{s_n,t})\}(\omega)\mathrm{d}t\Big|\\
			&\leq \ve/2+K_{\ve}\sum\limits_{n=m}^{\infty}n2^{-n/6}\leq\ve/2+\ve/2=\ve.
		\end{align*}
		Define the set $D_r$ by
		\begin{align*}
			D_r:=\Big\{\omega\in\Omega:\,\sqrt{s}\int_0^1\varphi_r(t,x+W_{s,t})(\omega)dt\leq\ve,\,\forall\,(s,x)\in]0,1]\times[-1,1]^d\Big\}.
		\end{align*}
		Then, one sees that $A_r\cap B_r\subset D_r$ and thus $\Pb(D_r)\geq1-\ve$. We set $D=\bigcap_{r\in\N}D_r$. Since $(\varphi_r)_{r\in\N}$ is non-decreasing, $(D_r)_{r\in\N}$ is non-increasing and as a consequence, $\Pb(D)\geq1-\ve$. Moreover, by the Beppo-Levy theorem 
		\begin{align*}
			\lim\limits_{r\to\infty}\sqrt{s}\int_0^1\varphi_r(t,x+W_{s,t})\mathrm{d}t=\sqrt{s}\int_0^1\mathbf{1}_U(t,x+W_{s,t})\mathrm{d}t
		\end{align*}
		and then
		\begin{align*}
			D\subset\Big\{\omega\in\Omega:\,\sqrt{s}\int_0^1\mathbf{1}_U(t,x+W_{s,t})\mathrm{d}t\leq\ve,\,\forall\,(s,x)\in]0,1]\times[-1,1]^d\Big\},
		\end{align*}
		which yields the desired result.
	\end{proof}
	
	The next result shows that the two-parameter Wiener process regularises the averaging operator $T^W_{[0,1]}[b]$ for any bounded Borel measurable function $b$.  
	\begin{lemma}\label{lem:NoiseRegular}
		Let $b$ satisfy conditions of Corollary \ref{corol:DavieSheetdd34} and let $(s_n,x_n)_{n\in\N}$ be a sequence in $]0,1]\times[-1,1]^d$ that converges to $(s,x)$, where $s>0$. Then 
		\begin{align}\label{eq:NoiseRegular}
			\lim\limits_{n\to\infty}\int_0^1b(t,x_n+W_{s_n,t})\mathrm{d}t=\int_0^1b(t,x+W_{s,t})\mathrm{d}t,\quad\Pb\text{-a.s.}
		\end{align}
	\end{lemma}
	\begin{remark}
		Notice that the above result which is key to prove the regularisation by noise is not valid if one replaces the Brownian sheet by any two dimensional continuous function. For example, consider the function $w:[0,1]^2\to\R$ defined by
			\begin{align*}
				w(s,t)=\left\{
				\begin{array}{cl}
					1-\exp\Big(-\dfrac{1}{t(1-s)}\Big)&\text{if }(s,t)\in[0,1[\times]0,1],\\
					1&\text{otherwise}.
				\end{array}
				\right.
			\end{align*}
			Then $w$ is continuous. Let $b$ be the usual integer part function on $[0,1]$ and consider the sequence $(s_n,n\in\N)$ defined by $s_n=1-1/n$. Then $\lim\limits_{n\to\infty}s_n=1$ and
			\begin{align*}
				\lim\limits_{n\to\infty}\int_0^1b(w(s_n,t))\mathrm{d}t=0\neq1=\int_0^1b(w(1,t))\mathrm{d}t.
			\end{align*}
	\end{remark}
	\begin{proof}[Proof of Lemma \ref{lem:NoiseRegular}.]

		For every $r\in\N$, set $\ve_r=2^{-r}$ and consider the corresponding $\eta_r:=\eta_{\ve_r}$ of Lemma \ref{lem:IntIndicFuncEstim}. By Lusin's theorem applied to each $r\in\N$, we can find a function $b_r\in\mathcal{C}_b([0,1]\times\R^d)$ and an open set $U_r\subset[0,1]\times\R^d$ such that
		\begin{align*}
			\Vert b_r\Vert_{\infty}\leq1,\quad|U_r|\leq\eta_r,\quad b_r(t,x)=b(t,x)\quad\text{for all }(t,x)\notin U_r.
		\end{align*}
		By Lemma \ref{lem:IntIndicFuncEstim}, there exists a subset $\Omega_r$ of $\Omega$ with $\Pb(\Omega_r)\geq1-\ve_r$ such that for any $(s,x)\in]0,1]\times[-1,1]^d$,
		\begin{align*}
			\int_0^1\mathbf{1}_{U_r}(t,x+W_{s,t})\mathrm{d}t\leq\frac{\ve_r}{\sqrt{s}},\quad\text{on }\Omega_r.
		\end{align*}	
		
		Next, observe that for any $r\in\N$,
		\begin{align*}
			\int_0^1b(t,x+W_{s,t})\mathrm{d}t= \int_0^1b\mathbf{1}_{U_r}(t,x+W_{s,t})\mathrm{d}t+\int_0^1b\mathbf{1}_{U^c_r}(t,x+W_{s,t})\mathrm{d}t
		\end{align*}
		and
		\begin{align*}
			\int_0^1b_r(t,x+W_{s,t})\mathrm{d}t= \int_0^1b_r\mathbf{1}_{U_r}(t,x+W_{s,t})\mathrm{d}t+\int_0^1b_r\mathbf{1}_{U^c_r}(t,x+W_{s,t})\mathrm{d}t,
		\end{align*}
		where $U_r^c=([0,1]\times\R^d)\setminus U_r$. Hence, we have
		\begin{align}
			\int_0^1b_r(t,x+W_{s,t})\mathrm{d}t-
			2\int_0^1\mathbf{1}_{U_r}(t,x+W_{s,t})\mathrm{d}t 
			&\leq\int_0^1b(t,x+W_{s,t})\mathrm{d}t\nonumber\\
			&\leq\int_0^1b_r(t,x+W_{s,t})\mathrm{d}t+
			2\int_0^1\mathbf{1}_{U_r}(t,x+W_{s,t})\mathrm{d}t.\label{eq:IneqlemRegula}
		\end{align}
		Now, let $(s_n,x_n)_{n\in\N}$ be a sequence in $[0,1]\times[-1,1]^d$ that converges to $(s,x)$. We deduce from \eqref{eq:IneqlemRegula} that, on $\Omega_r$,
		\begin{align*}
			\int_0^1b(t,x+W_{s,t})\mathrm{d}t-\frac{4\ve_r}{\sqrt{s}}
			\leq\int_0^1b_r(t,x+W_{s,t})\mathrm{d}t-\frac{2\ve_r}{\sqrt{s}}\leq\liminf\limits_{n\to+\infty} \int_0^1b(t,x_n+W_{s_n,t})\mathrm{d}t
		\end{align*}
		and
		\begin{align*}
			\int_0^1b(t,x+W_{s,t})\mathrm{d}t+\frac{4\ve_r}{\sqrt{s}}
			\geq \int_0^1b_r(t,x+W_{s,t})\mathrm{d}t+\frac{2\ve_r}{\sqrt{s}}\geq\limsup\limits_{n\to+\infty} \int_0^1b(t,x_n+W_{s_n,t})\mathrm{d}t.
		\end{align*}	
		Define
		\begin{align*}
			\Omega_{\infty}:=\liminf\limits_{r\to\infty}\Omega_r=\bigcup\limits_{q=1}^{\infty}\bigcap\limits_{r=q}^{\infty}\Omega_r.
		\end{align*}
		Then since $\Pb(\Omega_r)\geq1-\ve_r$ and $\sum\limits_{r=1}^{\infty}\ve_r=1$, it follows from the Borel-Cantelli lemma that $\Pb(\Omega_{\infty})=1$. Moreover, on $\Omega_{\infty}$, one has
		\begin{align*}
			\limsup\limits_{n\to+\infty} \int_0^1b(t,x_n+W_{s_n,t})\mathrm{d}t
			\leq \int_0^1b(t,x+W_{s,t})\mathrm{d}t\leq\liminf\limits_{n\to+\infty} \int_0^1b(t,x_n+W_{s_n,t})\mathrm{d}t.
		\end{align*}
		This completes the proof.
	\end{proof}
	\begin{proof}[Proof of Theorem \ref{theo:Pseudometric}.]
		We deduce from Lemma \ref{lem:PseudoMetric1} that for every $r\in\N$, there exist a positive deterministic constant $C_{r}$ such that for any Borel measurable function $b$ satisfying conditions of Corollary \ref{corol:DavieSheetdd34}, one can find $\mathcal{N}_{r}\subset\Omega$ with $\Pb(\mathcal{N}_r)<2^{-r-1}$ and  
		\begin{align*}
			\rho_{nk}(s,x;\ps,\xp)(\omega)\leq \frac{C_r2^{-n/2}}{\sqrt{s}}\Big[n+ \log^+\frac{1}{|\xp-x|+\sqrt{\ps-s}}\Big]\Big(|\xp-x|+\sqrt{\ps-s}\Big),
		\end{align*}
		for every $\omega\in\Omega\setminus\mathcal{N}_{r}$, every 
		$(n,k)\in\N^2$, $0\leq k\leq2^n-1$ and every  $(s,\ps,x,\xp)\in\mathcal{Q}$.	
		We define
		\begin{align*}
			\mathcal{N}_{\infty}=\limsup\limits_{r\to\infty}\mathcal{N}_r=\bigcap\limits_{r=1}^{\infty}\bigcup\limits_{\ell=r}^{\infty}\mathcal{N}_{\ell}.
		\end{align*}
		Since $\sum\limits_{r=0}^{\infty}\Pb(\mathcal{N}_r)<1$, then $\Pb(\mathcal{N}_{\infty})=0$ and for any $\omega\in\Omega_2=\Omega\setminus\mathcal{N}_{\infty}$, there exist $r_{\omega}\in\N$ such that $\omega\in\Omega\setminus\mathcal{N}_{\ell}$, for all $\ell\geq r_{\omega}$. In particular for a given Borel measurable function $b$ satisfying conditions of Corollary \ref{corol:DavieSheetdd34},
		\begin{align}\label{eqlemma6}
			\rho_{nk}(s,x;\ps,\xp)(\omega)\leq \frac{C_{r_{\omega}}2^{-n/2}}{\sqrt{s}}\Big[n+ \log^+\frac{1}{|\xp-x|+\sqrt{\ps-s}}\Big]\Big(|\xp-x|+\sqrt{\ps-s}\Big)
		\end{align}
		for every $\omega\in\Omega_2$, every $(s,\ps,x,\xp)\in\mathcal{Q}$ and every $(n,k)\in\N^2$, $0\leq k\leq2^n-1$.
		Now, fix $\omega\in\Omega_2$, $n\in\N$, $k\in\{0,1,\cdots,2^n-1\}$,  $(s,\ps,x,\xp)\in]0,1]^2\times[-1,1]^d$ with $s\leq\ps,\,(s,x)\neq(\ps,\xp)$. Let $(s_q,x_q)_{q\in\N}$ (respectively $(\ps_q,\xp_q,q\in\N)$) be a sequence of dyadic numbers in $[0,1]\times[-1,1]^d$ that converges to $(s,x)$ (respectively $(\ps,\xp)$).  Suppose without loss of generality that for every $q\in\N$, $0<s_q<s_{q+1}$ and $\ps_{q+1}<\ps_q$. Using \eqref{eqlemma6},  we have
		\begin{align*}
			\rho_{nk}(s_q,x_q;\ps_q,\xp_q)(\omega)\leq \frac{C_{r_{\omega}}2^{-n/2}}{\sqrt{s_q}}\Big[n+ \log^+\frac{1}{|\xp_q-x_q|+\sqrt{\ps_q-s_q}}\Big]\Big(|\xp_q-x_q|+\sqrt{\ps_q-s_q}\Big)
		\end{align*}
		for any $q\in\N$, any $\omega\in\Omega_2$, any $b$ and any $(n,k)$.
		We deduce from Lemma \ref{lem:NoiseRegular} that there also exists $\Omega_{\infty}\subset\Omega$ with $\Pb(\Omega_{\infty})=1$ such that for any $\omega\in\Omega_{\infty}$,
		\begin{align*}
			\lim\limits_{q\to\infty}\rho_{nk}(s_q,x_q;\ps_q,\xp_q)(\omega)=\rho_{nk}(s,x;\ps,\xp)(\omega).
		\end{align*}
		Hence for any $\omega\in\Omega_0=\Omega_2\cap\Omega_{\infty}$,
		\begin{align*}
			\rho_{nk}(s,x;\ps,\xp)(\omega)&=\lim\limits_{q\to\infty}\rho_{nk}(s_q,x_q;\ps_q,\xp_q)(\omega)\\&\leq\lim\limits_{q\to\infty} \frac{C_{r_{\omega}}2^{-n/2}}{\sqrt{s_q}}\Big[n+ \log^+\frac{1}{|\xp_q-x_q|+\sqrt{\ps_q-s_q}}\Big]\Big(|\xp_q-x_q|+\sqrt{\ps_q-s_q}\Big)\\
			&=\frac{C_{r_{\omega}}2^{-n/2}}{\sqrt{s}}\Big[n+ \log^+\frac{1}{|\xp-x|+\sqrt{\ps-s}}\Big]\Big(|\xp-x|+\sqrt{\ps-s}\Big).
		\end{align*}
		The proof is then completed by taking $C_0(\omega)=C_{r_{\omega}}$ since $\Pb(\Omega_0)=1$.	 
	\end{proof}

	\begin{proof}[Proof of Corollary \ref{corol:Pseudometric}.] 
		Let $I=(a,\pa)$, with $0\leq a<\pa\leq1$. For every $\gamma\in[0,1]$, set $T^W_{\gamma}=T^W_{[0,\gamma]}$. Suppose first that $(a,\pa)$ is a couple of dyadic numbers. Let $n$ be the smallest non-negative integer such that $2^{-n-1}\leq |\pa-a|$. For every $\ell\geq n$, choose $a_{\ell}$ (respectively $\pa_{\ell}$) to minimise $|a-a_{\ell}|$ (respectively $|\pa-\pa_{\ell}|$) under the constraint $2^{\ell}a_{\ell}\in\Z$ (respectively $2^{\ell}\pa_{\ell}\in\Z$). Then $(a_n,\pa_n)$ is a couple of real numbers that are either equal or dyadic neighbours. A similar assertion is valid for  $(a_{\ell},a_{\ell+1})$ and $(\pa_{\ell},\pa_{\ell+1})$. Since $(a_{\ell},\ell\geq n)$ (respectively $(\pa_{\ell},\ell\geq n)$) converges to $a$ (respectively $\pa$) as $\ell$ goes to infinity, we have
		\begin{align*}
			\left|T^W_{I}[b](\ps,\xp)-T^W_{I}[b](s,x)\right|&=\left|T^W_{\pa}[b](\ps,\xp)-T^W_{\pa}[b](s,x)-T^W_{a}[b](\ps,\xp)+T^W_{a}[b](s,x)\right|\\
			&\leq\left|T^W_{\pa_n}[b](\ps,\xp)-T^W_{\pa_n}[b](s,x)-T^W_{a_n}[b](\ps,\xp)+T^W_{a_n}[b](s,x)\right|\\
			&\quad+\sum\limits_{\ell=n}^{\infty}\left|T^W_{a_{\ell}}[b](\ps,\xp)-T^W_{a_{\ell}}[b](s,x)-T^W_{a_{\ell+1}}[b](\ps,\xp)+T^W_{a_{\ell+1}}[b](s,x)\right|\\
			&\quad+\sum\limits_{\ell=n}^{\infty}\left|T^W_{\pa_{\ell}}[b](\ps,\xp)-T^W_{\pa_{\ell}}[b](s,x)-T^W_{\pa_{\ell+1}}[b](\ps,\xp)+T^W_{\pa_{\ell+1}}[b](s,x)\right|.
		\end{align*}
		We deduce from Theorem \ref{theo:Pseudometric} that
		\begin{align*}
			&\left|T^W_{\pa_n}[b](\ps,\xp)-T^W_{\pa_n}[b](s,x)-T^W_{a_n}[b](\ps,\xp)+T^W_{a_n}[b](s,x)\right|\\
			&\leq \frac{C_02^{-n/2}}{\sqrt{s}}\Big[n+\log^+\frac{1}{|\xp-x|+\sqrt{\ps-s}} \Big]\Big(|\xp-x|+\sqrt{\ps-s}\Big)
		\end{align*}
		and, for every $\ell\geq n$,
		\begin{align*}
			&\left|T^W_{a_{\ell}}[b](\ps,\xp)-T^W_{a_{\ell}}[b](s,x)-T^W_{a_{\ell+1}}[b](\ps,\xp)+T^W_{a_{\ell+1}}[b](s,x)\right|\\&\leq \frac{C_02^{-(\ell+1)/2}}{\sqrt{s}}\Big[\ell+1+\log^+\frac{1}{|\xp-x|+\sqrt{\ps-s}} \Big]\Big(|\xp-x|+\sqrt{\ps-s}\Big)
		\end{align*}
		and
		\begin{align*}
			&\left|T^W_{\pa_{\ell}}[b](\ps,\xp)-T^W_{\pa_{\ell}}[b](s,x)-T^W_{\pa_{\ell+1}}[b](\ps,\xp)+T^W_{\pa_{\ell+1}}[b](s,x)\right|\\&\leq \frac{C_02^{-(\ell+1)/2}}{\sqrt{s}}\Big[\ell+1+\log^+\frac{1}{|\xp-x|+\sqrt{\ps-s}} \Big]\Big(|\xp-x|+\sqrt{\ps-s}\Big).
		\end{align*}
		Thus, as $\sum\limits_{\ell=n+1}^{\infty}2^{-\ell/2}=(\sqrt{2}+1)2^{-n/2}$ and $\sum\limits_{\ell=n+1}^{\infty}(\ell+1)2^{-\ell/2}=(\sqrt{2}+1)(3+\sqrt{2}+n)2^{-n/2}$
		\begin{align*}
			&\left|T^W_{I}[b](\ps,\xp)-T^W_{I}[b](s,x)\right|\\&\leq \frac{C_02^{-n/2}}{\sqrt{s}}\Big[n+\log^+\frac{1}{|\xp-x|+\sqrt{\ps-s}} \Big]\Big(|\xp-x|+\sqrt{\ps-s}\Big)\\&\quad+\frac{2C_0}{\sqrt{s}}\Big(|\xp-x|+\sqrt{\ps-s}\Big)\sum\limits_{\ell=n}^{\infty}2^{-(\ell+1)/2}\Big[\ell+1+\log^+\frac{1}{|\xp-x|+\sqrt{\ps-s}} \Big]\\
			&\leq\frac{31C_02^{-n/2}}{\sqrt{s}}\Big[1+n+\log^+\frac{1}{|\xp-x|+\sqrt{\ps-s}} \Big]\Big(|\xp-x|+\sqrt{\ps-s}\Big)\\
			&\leq\frac{64C_0\sqrt{|\pa-a|}}{\sqrt{s}}\Big[1+\log^+\frac{1}{\pa-a}+\log^+\frac{1}{|\xp-x|+\sqrt{\ps-s}} \Big]\Big(|\xp-x|+\sqrt{\ps-s}\Big).
		\end{align*}
		The desired inequality follows by taking $C_1=64C_0$. The result in the general case follows from the continuity of the map $\gamma\longmapsto T^W_\gamma[b]$.
	\end{proof}

	
\end{document}